\documentclass[12pt]{amsart}

\usepackage{amssymb, latexsym}

\def\ov{\overline}

\def\mb{\mathbb}
\def\mr{\mathrm}

\def\mbSn{\mathbf{S}_n}
\def\mbS{\mathbf{S}}
\def\mbR{\mathbb{R}}
\def\mbZ{\mathbb{Z}}

\def\mbD{\mathbb{D}}

\newcommand{\mc}{\mathcal}

\newcommand{\oW}{\overline{W}}

\theoremstyle{plain}
\newtheorem{theorem}{Theorem}[section]
\newtheorem{corollary}[theorem]{Corollary}
\newtheorem{proposition}[theorem]{Proposition}
\newtheorem{lemma}[theorem]{Lemma}

\theoremstyle{definition}
\newtheorem{definition}[theorem]{Definition}
\newtheorem{remark}[theorem]{Remark}
\newtheorem{example}[theorem]{Example}

\numberwithin{equation}{section}

\usepackage
[hypertex]
{hyperref}

\usepackage[dvipsnames]{color}

\begin{document}


\title[Finitely generated dyadic convex sets]{Finitely generated dyadic convex sets}

\author[Matczak]{K. Matczak$^1$}
\address{$^1$ Faculty of Civil Engineering, Mechanics and
Petrochemistry in P\l ock\\
Warsaw University of Technology\\
09-400 P\l ock, Poland}

\author[Mu\'{c}ka]{A. Mu\'{c}ka$^2$}
\address{$^2$ Faculty of Mathematics and Information Sciences\\
Warsaw University of Technology\\
00-662 Warsaw, Poland}

\author[Romanowska]{A.B. Romanowska$^3$}
\address{$^3$ Faculty of Mathematics and Information Sciences\\
Warsaw University of Technology\\
00-662 Warsaw, Poland}

\email{$^1$matczak.katarzyna@gmail.com\phantom{,}}
\email{$^2$Anna.Mucka@pw.edu.pl\phantom{,}}
\email{$^3$Anna.Romanowska@pw.edu.pl\phantom{,}}

\keywords{dyadic rational numbers, dyadic affine space, dyadic convex set, dyadic polytope, commutative binary mode, generating dyadic polytopes}

\subjclass[2010]{20N02, 08A05, 52B11, 52A01}

\thanks{\emph{Please cite as}:
Matczak, K.; Mućka, A.; Romanowska, A. B.; Finitely generated dyadic convex sets;
Internat. J. Algebra Comput. 33 (2023), no. 3, 585–615.
}

\date{March 2023}

\begin{abstract}

Dyadic rationals are rationals whose denominator is a power of $2$. We define \emph{dyadic $n$-dimensional convex sets} as the intersections with $n$-dimensional dyadic space of an $n$-dimensional real convex set. Such a dyadic convex set is said to be
a \emph{dyadic $n$-dimensional polytope} if the real convex set is a polytope whose vertices lie in the dyadic space.
Dyadic convex sets are described as subreducts (subalgebras of reducts) of certain faithful affine spaces over the ring of dyadic numbers, or equivalently as commutative, entropic and idempotent groupoids under the binary operation of arithmetic mean.

The paper contains two main results. First, it is proved that, while all dyadic polytopes are finitely generated, only dyadic simplices are generated by their vertices. This answers a question formulated in an earlier paper. Then, a characterization of finitely generated subgroupoids of dyadic convex sets is provided, and it is shown how to use the characterization to determine the minimal number of generators of certain convex subsets of the dyadic plane.

\end{abstract}

\maketitle


\section{Introduction}\label{S:I}

The objects we study belong to both algebra and geometry. As algebras, they have one binary operation which is idempotent, commutative and entropic. As geometric objects, they are subsets of certain affine spaces.

Recall that \emph{dyadic rationals} are rationals whose denominator is a power of $2$. They form the principal ideal subdomain $\mbD = \mbZ[1/2]$ of the ring $\mbR$ of real numbers. The affine spaces of interest are certain faithful affine spaces over the ring $\mbD$ (or affine $\mbD$-spaces), i.e.,  idempotent reducts of certain faithful $\mbD$-modules. Affine spaces over the ring $\mbR$ (\emph{affine $\mbR$-spaces}) may be described as abstract algebras with uncountably many binary affine combinations indexed by the elements of $\mbR$. Affine $\mbD$-spaces are subreducts (subalgebras of reducts) of affine $\mbR$-spaces. The fact that $\mbD$ is a linearly ordered commutative ring with a nontrivial open unit interval makes it possible to define convex subsets of affine $\mbD$-spaces in a reasonable way. For a given dimension $n$, a subset $D$ of $\mbD^n$ is called a \emph{geometric dyadic convex set} (or briefly a \emph{dyadic convex set}) if it is the intersection of a convex subset of $\mbR^n$ with the space $\mbD^n \subseteq \mbR^n$. Such sets are also described as being \emph{convex relative to $\mbD$} (cf. for instance~\cite{B05}.) Then \emph{dyadic polytopes} are defined as dyadic convex sets obtained as the intersection of dyadic spaces with real polytopes whose vertices are contained in the dyadic space. In particular, \emph{dyadic polygons} are the intersections of the dyadic plane with real polygons spanned by vertices in the dyadic plane. The one-dimensional analogs are \emph{dyadic intervals}.
In this paper the name of a dyadic interval will be always used to denote an interval which is a dyadic polytope.

Real convex sets are described algebraically as certain barycentric algebras, subsets of $\mbR^n$ closed under operations of weighted means with weights from the open unit interval $I^{\circ} = \, ]0,1[$ of $\mbR$. Similarly, dyadic convex sets may be described as subsets of $\mbD^n$  closed under weighted means with weights from the open dyadic unit interval $I^{\circ} \cap \mbD$. However, as shown already by Je\v{z}ek and Kepka~\cite{JK76}, the operations determined by $I^{\circ} \cap \mbD$ are all generated by the single arithmetic mean operation
\[
x \circ y := xy \underline{1/2} = \frac{1}{2}(x + y).
\]
The arithmetic mean is an idempotent, commutative and entropic operation.
This fact allows one to define dyadic convex sets equivalently as grou\-po\-ids with the algebraic structure of a commutative binary mode (or briefly $\mc{CB}$-modes), i.e. algebras with one binary idempotent, commutative and entropic operation. Such algebras have a well developed theory. (See e.g.~\cite{JK83, MMR19, MMR19a, MR04, MR05, MRS11, RS85, RS02}.)

Affine spaces and their subreducts, as well as $\mc{CB}$-modes, belong to the class of \emph{modes}, idempotent and entropic algebras. (For the general theory of modes, see the monographs~\cite{RS85, RS02}.)

Dyadic and real convex sets share many properties. However, there are also essential differences. For example, real convex sets are characterized as cancellative barycentric algebras (see~\cite[Thm.~5.8.6]{RS02}, but not all cancellative $\mc{CB}$-modes are dyadic convex sets (see \cite[\S~5]{MR04},~\cite[\S~1.2]{CR13}.) While real intervals belong to only finitely many isomorphism classes of barycentric algebras, there are infinitely many pairwise non-isomorphic dyadic intervals.
As a subgroupoid of $(\mbD,\circ)$, each nontrivial interval of $\mbD$ is isomorphic to some dyadic interval $\mbD_k = [0,k]$, where $k$ is an odd positive integer. Two such intervals are isomorphic precisely when their right hand ends are equal~\cite[\S~3]{MRS11}.

The problem of locating generating elements of dyadic polygons was first considered in~\cite{MRS11}.
While real intervals which are polytopes are minimally generated by their two ends, the dyadic intervals are minimally generated by $2$ or $3$ elements.
The interval $\mbD_{1}$ is generated by $0$ and $1$, and for $k > 1$, the interval $\mbD_{k}$ is generated by three, but no fewer, elements. For each $d \in \mbD$, the dyadic intervals $[d,d+2^{n}]$ are isomorphic to $\mbD_1$, and the dyadic intervals $[d,d+k2^{n}]$ are isomorphic to $\mbD_k$. More generally, it was shown that each dyadic triangle is finitely generated, and then that each dyadic polygon is finitely generated. However, the methods used in that paper were not sufficient to determine the exact minimal number of generators. Moreover, the following question was left open: Are all dyadic polytopes finitely generated? (See~\cite[Problem~1.1]{MRS11}.)
The question arose in connection with our investigations of quasivarieties of $\mc{CB}$-modes~\cite{MR04, MR05} and of duality for dyadic polytopes ~\cite{MMR19, MMR19a}. A related question asks whether finitely generated subgroupoids of $(\mb{D}^n, \circ)$ are isomorphic to dyadic polytopes. A partial answer was then given in~\cite{MMR19}: It was shown that each finitely generated subgroupoid of $(\mb{D}, \circ)$ is isomorphic to a dyadic interval.

In this paper we provide an affirmative solution to the problem formulated in~\cite{MRS11}, and show that each dyadic polytope is finitely generated (Theorem~\ref{P:fingen}). However, there are some other subsets of $\mbD^n$, closed under the operation $\circ$, which are not isomorphic to dyadic polytopes but are finitely generated
(Example~\ref{Ex:notdpol}). To characterize finitely generated subgroupoids of $\mbD^n$, we introduce a new class of so-called semipolytopes. A \emph{semipolytope} is a subgroupoid of
a dyadic polytope such that both have the same vertices and the same interior (Definition~\ref{D:qpol}).
The main result, Theorem~\ref{T:main}, states that a subgroupoid $(D, \circ)$ of the groupoid $(\mbD^n, \circ)$ is finitely generated precisely when it is isomorphic to a semipolytope.
This result is then used in the final section to calculate the minimal number of generators of certain polygons and semipolygons.

All dyadic convex sets are considered in this paper as $\mathcal{CB}$-modes. We use notation, terminology and conventions similar to those of~\cite{RS02} and the previously referenced papers. In particular, the reader may wish to consult the papers~\cite{CR13},~\cite{MMR19} and~\cite{MRS11}. For more details and information on affine spaces, convex sets and barycentric algebras, we also refer the reader to the monographs~\cite{RS85, RS02} and the recent survey~\cite{R18}. For convex polytopes, see the first two chapters of~\cite{AB83}.

\section{Subsets of dyadic affine spaces}\label{S:2}

In this section we recall and augment the basic facts concerning dyadic affine spaces and dyadic convex sets provided in~\cite{CR13},~\cite{MR04} and~\cite{MRS11}.

\subsection{Dyadic affine spaces}\label{S:2.1}

First recall that affine spaces over a commutative ring $R$ (\emph{affine $R$-spaces}), where $R$ is $\mbR$ or $\mbD$, can be considered as the reducts $(A,\underline{R})$ of $R$-modules $(A,+,R)$, where $\underline{R}$ is the set of binary affine combinations
\begin{equation}\label{E:afoper}
ab\,\underline{r} = a(1-r)+br \,
\end{equation}
for all $r \in R$ and $a,b \in A$. (See~\cite[\S~5.3, \S~6.3]{RS02}.) In particular,
affine spaces over the ring $\mbD$ (affine $\mbD$-spaces) are considered here as algebras $(A,\underline{\mbD})$ with the set $\underline{\mbD} = \{\underline{d} \mid d \in \mbD\}$ of basic operations. The class of all affine $R$-spaces forms a variety~\cite{C75}.

In this paper, we are especially interested in subreducts of certain faithful affine $\mbD$-spaces. (Recall that in a faithful affine $\mbD$-space, the two operations $\underline{d}$ and $\underline{e}$ for $d\ne e\in \mbD$ are different.) By results of~\cite{MR04} and~\cite{CR13}, it follows that one of the minimal quasivarieties of affine $\mbD$-spaces is the quasivariety $\mathsf{Q_A}(\mbD)$ of affine $\mbD$-spaces, generated by the affine $\mbD$-space $\mbD$.
Finitely generated members of $\mathsf{Q_A}(\mbD)$ are free affine $\mbD$-spaces on finitely many generators. The free affine $\mbD$-space on $n+1$\\ generators, for
$n = 0, 1, \dots$, is $\mbD^{n}$. The number $n$ is called the \emph{dimension} of the
affine $\mbD$-space $\mbD^{n}$. The set of elements of the free affine $\mbD$-space $(x_0, \dots, x_n)\mbD$ on $n+1$ generators $x_0, \dots, x_n$ is given as
\[
(x_0, \dots, x_n)\mbD = \Big\{x_0 r_0 + \dots + x_n r_n\, \Big|\, r_i \in \mbD, \sum_{i=0}^n r_i = 1\Big\}.
\]
(See~\cite{PRS03}.)
All members of $\mathsf{Q_A}(\mbD)$ are faithful. Note, however, that there are faithful affine $\mbD$-spaces not contained in $\mathsf{Q_A}(\mbD)$, as for example the affine space $\mbD \times \mbD / (3) \cong \mbD \times \mb{Z}_{3}$.
\footnote{In this context, the explanation given in~\cite{CR13, MR05, MMR19} may require further clarification.}

Unlike the case of the affine $\mbR$-space $\mbR^n$, the affine $\mbD$-space $\mbD^n$ contains infinitely many $n$-dimensional affine $\mbD$-subspaces. First note that nontrivial subspaces of the affine $\mbD$-space $\mbD$ have the form
\[
m \mbD = \{md \mid d \in \mbD\}
\]
for some positive integer $m$. Note also that $2^k \mbD = \mbD$ and $2^k (2l + 1) \mbD = (2l + 1) \mbD$.

Nontrivial affine subspaces of affine $\mbD$-spaces $\mbD^n$ for $n \geq 2$ are more difficult to describe. First note that
\[
m_1 \mbD \times \dots \times m_n \mbD,
\]
where $m_1, \dots, m_n$ are positive integers, are $n$-dimensional subspaces of the affine $\mbD$-spaces $\mbD^n$. Then, for $1 \leq k \leq n$, the products
\[
m_{i_1} \mbD \times \dots \times m_{i_k} \mbD,
\]
are $k$-dimensional subspaces of $\mbD^n$, and hence are isomorphic to the affine $\mbD$-space $\mbD^k$.
However, there exist other affine subspaces of $\mbD^n$. Examples are given by proper subspaces of dyadic lines in $\mbD^n$. For example, the set $\{(3d,3d) \mid d \in \mbD\}$ is a subspace of the dyadic line $\{(d,d) \mid d \in \mbD\}$, a one-dimensional subspace of $\mbD^2$.

Let $X$ be a subset of a finite-dimensional affine $\mbD$-space $A$ from the quasivariety $\mathsf{Q_A}(\mbD)$.
Then the \emph{affine $\mbD$-hull} $\mr{aff}_{\mbD}(X)$ of $X$ is
the intersection of all the affine subspaces of $A$ that contain $X$. The \emph{affine $\mbR$-hull} $\mr{aff}_{\mbR}(X)$ is the intersection of the all affine $\mbR$-spaces containing~$X$.

\subsection{Dyadic convex sets}\label{S:2.2}

\begin{definition}\cite{CR13}
A subgroupoid $(B, \circ)$ of the reduct $(A, \circ)$ of an affine $\mbD$-space $A$ from $\mathsf{Q_A}(\mbD)$ is called an \emph{algebraic dyadic convex subset of} $A$ (or briefly an \emph{algebraic dyadic convex set}).
\end{definition}

Isomorphic copies of algebraic dyadic convex sets form a (minimal) subquasivariety $\mathsf{Q}(\mbD)$ of the variety $\mc{CBM}$ of $\mc{CB}$-modes \cite[\S~3]{MR04}.

\begin{definition}\cite{CR13, MRS11}
An algebraic dyadic convex subset of $\mbD^n$ is \emph{geometric}
if it is the intersection of a convex subset $C$ of $\mbR^n$  with the subspace~$\mbD^n$.
\end{definition}

Not all algebraic dyadic convex sets are geometric. For example, $3 \mbD$ is an algebraic but not a geometric dyadic convex set.
If not otherwise indicated, the term ``dyadic convex set'' will normally be used to denote geometric dyadic convex sets. It is often convenient to use the subspace topology on a dyadic subspace of $\mbR^n$ coming from the Euclidean topology on $\mbR^n$. We use this topology when speaking about closed or open dyadic convex sets.

\begin{definition}
If $C \in \mathsf{Q}(\mbD)$, and $\mr{aff}_{\mbD}(C)$
is of finite dimension $n$, then we say that $C$ is
\emph{finite-dimensional}, and that its \emph{dimension} $\dim(C)$ equals $n$.
\end{definition}

If $C$ is an algebraic dyadic convex subset of an affine $\mbD$-space $\mbD^n$, then the \emph{convex $\mbR$-hull} $\mr{conv}_{\mbR}(C)$ of $C$ in $\mbR^n$ is the intersection of all convex subsets of $\mbR^n$ containing $C$. Then $\mr{conv}_{\mbR}(C)$ may be considered as the subalgebra generated by $C$ in the real convex set $(\mbR^n, \underline{I}^\circ)$ , where $\underline{I}^\circ$ is the set of operations $\underline{r}$ defined by~\eqref{E:afoper} for $r \in I^{\circ}$. The \emph{convex $\mbD$-hull} ~$\mr{conv}_{\mbD}(C)$ of $C$ in $\mbD^n$ is the intersection of $\mr{conv}_{\mbR}(C)$ with $\mbD^n$,
\begin{equation}\label{E:convhull}
\mr{conv}_{\mbD}(C) = \mr{conv}_{\mbR}(C) \cap \mbD^n,
\end{equation}
and is obviously geometric. An algebraic $n$-dimensional convex subset $C$  of an affine space
$\mbD^n$, where $n$ is a positive integer, is a geometric dyadic convex subset of $\mbD^n$ precisely if
\begin{equation}\label{E:geom}
C = \mr{conv}_{\mbD}(C),
\end{equation}
or in other words, if $C$ consists of all dyadic points of the real convex set generated by $C$. The convex $\mbR$-hull of a subset $X$ of $\mbD^n$ is defined similarly as for (algebraic) dyadic convex sets.

The \emph{relative interior} $\mr{int}(C)$ of an algebraic convex subset $C$ of $\mbD^n$ (or briefly the \emph{interior} of $C$) is the intersection with $\mr{aff}_{\mbD}(C)$ of the relative interior of $\mr{conv}_{\mbR}(C)$ in the affine $\mbR$-hull of $C$. The elements of $\mr{int}(C)$ are called \emph{interior points} of $C$. Note that if $C$ is geometric and $n$-dimensional, then $\mr{aff}_{\mbD}(C) = \mbD^n$.

\begin{definition}\cite{MRS11}
A \emph{dyadic $n$-dimensional polytope} is the intersection with the dyadic space $\mbD^n$ of an $n$-dimensional real polytope whose vertices lie in the dyadic space.
\end{definition}

Dyadic polytopes are obviously geometric.
If $P$ is a dyadic $n$-dimen\-sional polytope, then there is a real $n$-dimensional polytope $C$ with dyadic vertices such that
\[
P = C \cap \mbD^n.
\]
Hence
\[
C = \mr{conv}_{\mbR}(P),
\]
and then,
\[
P = C \cap \mbD^n = \mr{conv}_{\mbR}(P) \cap \mbD^n.
\]
On the other hand, by \eqref{E:convhull} and \eqref{E:geom},
\begin{equation}\label{E:pol}
P = \mr{conv}_{\mbR}(P) \cap \mbD^n = \mr{conv}_{\mbD}(P).
\end{equation}
The vertices of the (real) polytope $\mr{conv}_{\mbR}(P)$ are contained in $P$, and will be referred to as to the vertices of $P$.

\subsection{Walls}\label{S:2.3}

For a barycentric algebra $(C,\underline{I}^o)$, which is a convex set, a subset $A$ of $C$ is a \emph{wall}, if
\[
\forall\ a,b\in C\,,\  \forall\,\ r \in I^o\,,\
a b\underline{r} \in A
\
\Leftrightarrow
\
a\in A
\mbox{ and }
b\in A.
\]
Similarly, for a dyadic convex set $(C,\circ)$, a subset $A$ of $C$ is a \emph{wall}, if
\[
\forall\ a,b\in C\,,\
a\circ b\in A
\
\Leftrightarrow
\
a\in A
\mbox{ and }
b\in A.
\]
(See e.g.~\cite[\S 3.3]{RS85} and~\cite[Def.~1.1.9]{RS02}.)
Note that in the case of real polytopes, the geometric concept of a face (see e.g. \cite[\S5, \S7]{AB83}) and the algebraic concept of a wall describe the same subsets of a polytope. The $0$-dimensional faces of a real $n$-dimensional polytope $P$ are its vertices, and the $1$-dimensional faces are the edges. The maximal (thus $(n-1)$-dimensional) faces are also called the \emph{facets} of $P$. The proper faces of a polytope are polytopes of smaller dimension, and each face of $P$ considered as a barycentric algebra is generated by some vertices of $P$. The \emph{boundary} of $P$ in the affine hull of $P$ is the difference between $P$ and its interior, and is the union of the proper faces of $P$.

Using the algebraic definition of a wall, one can easily extend the geometric concept of a face of a real polytope to the case of a dyadic polytope.
This extension preserves the basic properties of faces of real polytopes.
In particular, walls of a dyadic polytope may be described as follows.

Let $P$ be an $n$-dimensional dyadic polytope and let $C = \mr{conv}_{\mbR}(P)$.
First note that by \eqref{E:pol},
\begin{equation}
P = C \cap \mbD^n,
\end{equation}
i.e. $P$ consists of all dyadic points of the real polytope $C$. Moreover, the dyadic polytope $P$ and the real polytope $C$ have the same (dyadic) vertices.

\begin{proposition}\label{P:drwalls}
Let $P$ be an $n$-dimensional dyadic polytope and let $C = \mr{conv}_{\mbR}(P)$. Then the intersection of $P$ with a wall of $C$ is a wall of $P$.
\end{proposition}
\begin{proof}
Let $W$ be a wall of $C$. To show that $W \cap P$ is a wall of $P$, let $d_1, d_2$ be any elements of $P$. Then $d = d_1 \circ d_2 \in W \cap P$, precisely when $d_1, d_2 \in W \cap P$. Thus $W \cap P$ is a wall of $P$.
\end{proof}

We will show that the only walls of a dyadic polytope $P$ are those described in Proposition~\ref{P:drwalls}.

\begin{lemma}\label{L:3walls}
If $P$ is a one-dimensional dyadic polytope, then $P$ is a dyadic interval and has precisely three non-empty walls, two zero-dimensional walls and one improper one-dimensional wall $P$.
\end{lemma}
\begin{proof}
A one-dimensional dyadic polytope $P$ is a dyadic interval, as described in Section~\ref{S:I}. Up to isomorphism, $P$ is an interval $\mbD_k$ of $\mbD \subseteq \mbR$, for some odd positive integer $k$. Then $C = \mr{conv}_{\mbR}(\mbD_k)$ is the real interval with vertices $0$ and $k$. Each of the vertices $0$ and $k$ forms a $0$-dimensional wall of $\mbD_k$ and of $C$. We will show that $\mbD_k$ has only one more wall which coincides with $\mbD_k$.

Let $d$ be any dyadic element of $C$ different from $0$ and $k$.
We are looking for the smallest wall $W$ of $\mbD_k$ containing $d$. First note that for each $x$ in $\mbD$, the reflection $x_d = 2d - x$ of $x$ in $d$ is also a dyadic number and $x \circ x_d = d$. If both $x$ and $x_d$ belong to $\mbD_k$, then $x$ and $x_d$ belong to $W$. In particular, if $d = k/2$, then each $x \in \mbD_k$ belongs to $W$, and $W$ coincides with $\mbD_k$.

Let $d$ be different from $0, k/2$ and $k$. We will show that $k/2 \in W$, which implies that $W$ coincides with $\mbD_k$. Without loss of generality it suffices to consider the case where $0 < d < k/2$. First let $k/4 \leq d < k/2$. Then $k/2 \leq 2d$, and for each $0 \leq x \leq d$, we have $d = x \circ x_d$. Hence $x, x_d \in W$. It follows that the dyadic interval $[0, 2d]$ is contained in $W$. In particular $k/2 \in W$.
Now let $0 < d < k/4$. Then there is a positive integer $i$ such that $2^i d < k/2 < 2^{i+1} d < k$. Similarly as before one shows that the dyadic intervals $[0,2d]$, and then $[0,4d], \dots, [0, 2^{i+1} d]$ are contained in $W$. Obviously, the last interval $[0, 2^{i+1} d]$ contains $k/2$.
\end{proof}

If $C$ is the smallest real polytope containing a dyadic polytope $D$, then we say that $D$ generates $C$. In other words, $C$ is generated by $D$ as a barycentric algebra.

\begin{proposition}\label{P:rdwalls}
Let $P$ be an $n$-dimensional dyadic polytope and let $C = \mr{conv}_{\mbR}(P)$. Then for each $k$-dimensional wall $W$ of $P$, where $k \leq n$, the real polytope $\overline{W} := \mr{conv}_{\mbR}(W)$ is a wall of $C$ generated by $W$, and $W = \overline{W} \cap P$.
\end{proposition}
\begin{proof}
Let $W$ be a wall of $P$, $\oW = \mr{conv}_{\mbR}(W)$, and $F$ be the wall of $C$ generated by the wall $W$. Note that
\[
W \subseteq \oW = \mr{conv}_{\mbR}(W)\subseteq F \subseteq C = \mr{conv}_{\mbR}(P).
\]
The proof is by induction on the dimension of a wall of $P$.
If $\dim(W) = 0$, then $W$ consists of one element, which is a vertex of both $F$ and $P$. If $\dim(W) = 1$, then $W \subseteq \overline{W} \subseteq F$, and all three $W, \oW$ and $F$ are one-dimensional. Moreover $F$ is a real interval with dyadic ends, which are vertices of $P$ and $C$, containing the dyadic interval $F \cap P$ of all dyadic points of $F$. By Proposition~\ref{P:drwalls}, $F \cap P$ is a wall of $P$.
By Lemma~\ref{L:3walls}, $F \cap P$ has two $0$-dimensional walls and one improper wall $F \cap P$. In particular each point of $W$, different from a vertex, generates $F \cap P$ as a wall of $P$. It follows that $W = F \cap P = F \cap \mbD$, and $F = \oW$.

Now assume that the statement of Proposition holds for all walls of dimension smaller than
$k < n$. We will show that it holds for $k$. So let $W$ be a wall of $P$ of dimension $k$.  Define $\oW$ and $F$ as before. Note that $F$ is a real polytope with dyadic vertices, and $F \cap P$ is a dyadic polytope with the same vertices as $F$.

Let $d \in F \cap P$. We are looking for a smallest wall of $F \cap P$ containing $d$.

If $d$ generates a wall $U$ of $P$ contained in $F$ of dimension smaller than $k$, then by induction hypotheses, $U = \overline{U} \cap P$, where $\overline{U} = \mr{conv}_{\mbR}(U)$ is a wall of $F$. So assume that $d$ is an interior point of $F$.
Recall that each interior point of $F$ generates the wall $F$. (See~\cite[Thm.~5.6]{AB83}.) Indeed, proper walls of $F$ do not contain interior points of $F$. Moreover, $W, \oW$ and $F$ are $k$-dimensional. Let $x$ be any dyadic point of $F \cap P$ and assume that its reflection $x_d$ in $d$ also belongs to $F \cap P$. Since $W$ is a wall of $P$ and $d = x \circ x_d$, it follows that if $d \in W$, then both $x$ and $x_d$ also belong to $W$. Consider the real line $l_d$ through $d$ and $x$. Then an argument similar to that in the proof of Lemma~\ref{L:3walls} shows that each dyadic interval contained in $l_d \cap P$ is contained in $W$. In particular, if $l_d \cap P$ contains a vertex of $P$, then this vertex belongs to $W$.
It follows that any dyadic point of $F \cap P$ belongs to $W$, and hence $W = F \cap P = F \cap \mbD$, and $F = \oW$.
\end{proof}

\begin{corollary}\label{C:wallsdp}
Each wall of a dyadic polytope is also a dyadic polytope.
\end{corollary}
\begin{proof}
By Propositions~\ref{P:drwalls} and~\ref{P:rdwalls}, a wall $W$ of an $n$-dimensional dyadic polytope $P$ is the intersection with $\mbD^n$ of the real polytope, namely the wall $F$ of $C$ generated by $W$.
\end{proof}

Propositions~~\ref{P:drwalls} and~\ref{P:rdwalls} show that there is a one-to-one correspondence between walls of $P$ and walls of $C$, and for each wall $W$ of $P$, one has $W = \overline{W} \cap P$. In particular, $W$ and $\overline{W}$ have the same (dyadic) vertices. It follows that
the walls of $P$ are determined by the vertices of $P$ they contain. Proper walls of $P$ are polytopes of a dimension smaller than the dimension of $P$.
Note as well that the lattices of walls of $P$ and $C = \mr{conv}_{\mbR}(P)$ are isomorphic.
Recall that dyadic subspaces of $\mbR^n$ are considered in this paper as topological spaces with  the subspace topology coming from the Euclidean topology on $\mbR^n$. Then similarly as in the case of real polytopes,
the (relative) interior of a dyadic polytope $P$ consists of the points of $P$ which are not contained in the proper walls of $P$, and the boundary of $P$ (in its affine $\mbD$-hull) is the difference between $P$ and its interior.

Corollary~\ref{C:wallsdp} may be extended to some algebraic dyadic convex sets.
Let $D$ be an $n$-dimensional algebraic dyadic convex set such that $E = \rm{conv}_{\mbD}(D)$ is a (dyadic) polytope. In this case,
each wall $W$ of $D$ is the intersection with $D$ of the wall of $E$ generated by $W$. The proof is similar as in the case of dyadic polytopes, with $\mbD$ replaced by $\rm{aff}_{\mbD}(D)$ and $C$ replaced by $E$.
Note that the vertices of $E$ belong to $D$, and the set of all these vertices is contained in each set of generators of $D$. If $D$ has $k$ vertices, then $k \geq n+1$. Note as well that $D$, $\mr{conv}_{\mbD}(D)$ and $\mr{conv}_{\mbR}(D)$ have the same dimension.

\begin{example}
To illustrate the concepts introduced above, consider the points $A = (0,0), B = (3,0)$ and $C = (0,1)$ of the plane $\mbD^2 \subseteq \mbR^2$. The closed real triangle $T_r$ with vertices $A$, $B$ and $C$ is a real polytope generated (as the barycentric algebra) by these vertices. The dyadic triangle $T_d = T_r \cap \mbD^2$ is a dyadic polytope with vertices $A$, $B$ and $C$, and with $\mr{conv}_{\mbR}(T_d) = T_r$. It consists of all dyadic points contained in $T_r$. Let us note that $A$, $B$ and $C$ are not sufficient to generate $T_d$ (as a groupoid under the operation $\circ$), since $A$ and $B$ generate only the dyadic points $(a,0)$ with $a$ divisible by $3$. However, $T_d$ is generated by $A$, $B$, $C$ and $D = (1,0)$. (See~\cite[Cor.~7.5]{MRS11}.) On the other hand, $A$, $B$ and $C$ generate the (algebraic) dyadic convex subset $T_a$ of $3 \mbD \times \mbD$, which is not geometric.

All the three triangles have the same vertices, which form $0$-dim\-ne\-sion\-al walls.
The real intervals $[A,B], [B,C], [A,C]$ are the proper nontrivial walls of $T_r$. Similarly, the dyadic intervals $[A,B], [B,C], [A,C]$ are the proper nontrivial walls of $T_d$.
The proper nontrivial walls of $T_a$ are obtained as the intersections of the corresponding walls of $T_d$ with $T_a$.
Note as well, that the open real or dyadic triangle $ABC$ has no vertices, is not a polytope and has no nontrivial proper walls. The dyadic open triangle $ABC$ is geometric. On the other hand, the open dyadic triangle $ABC$ together with the dyadic interval $[A,C]$ is also geometric, is not a polytope, but has two vertices $A$ and $C$, and three proper walls $A$, $C$ and $[A,C]$.
\end{example}

\subsection{Simplices}\label{S:2.4}

Examples of dyadic polytopes generated by their vertices are given by some models of finitely generated free commutative binary modes.
First recall that in affine geometry, a (real) $n$-dimensional simplex is usually defined as a polytope with $n+1$ affinely independent vertices. It is known that proper faces of a simplex are simplices of smaller dimension~\cite[Thm.~12.1, Thm.~12.2]{AB83}. It was shown by W. Neumann~\cite{N70} that the free barycentric algebra on $n+1$ generators is the $n$-dimensional simplex with its vertices as free generators. The elements of the free barycentric algebra $X\mc{B}$ on $X$, where $X = \{x_0, x_1, \dots ,x_n\}$, are described as follows
\[
\Big\{x_0 r_0 + \dots + x_n r_n\, \Big|\, r_i \in I, \ \sum_{i=0}^n r_i = 1\Big\}.
\]
Here $I = [0,1]$ is the closed unit interval of $\mbR$. (See~\cite[L.~6.1, L.~6.2]{PRS03}).

If $S$ is an $n$-dimensional real simplex with dyadic vertices $x_0,
\dots, x_n$, then the subgroupoid of $(S,\circ)$ generated by these vertices is called an \emph{$n$-dimensional algebraic dyadic simplex}. Its elements form the set
\begin{equation}\label{E:dyadsimpl}
\Big\{x_0 r_0 + \dots + x_n r_n\, \Big|\, r_i \in \mbD_1 = I \cap \mbD,\ \sum_{i=0}^n r_i = 1\Big\}.
\end{equation}
Note that such a simplex is not necessarily geometric.
As proved in~\cite[L.~6.1, L.~6.2]{PRS03}, the free commutative binary mode $X\mc{CBM}$ on a set $X = \{x_0, \dots ,x_n\}$ is isomorphic with the $\underline{1/2}$-reduct generated by $X$ of the free affine $\mbD$-space $X\mbD$ on $X$. The elements of $X\mc{CBM}$ are described by~\eqref{E:dyadsimpl}.

It was shown earlier in~\cite{JK76} and~\cite[\S4.4]{RS85} that by taking
the elements
\begin{equation}\label{E:freegens}
e_0 = (0, \dots, 0), e_1 = (1, 0, \dots, 0), \dots, e_n = (0, \dots, 0, 1)
\end{equation}
as the generators $x_i$ in~\eqref{E:dyadsimpl}, one obtains the (geometric) dyadic simplex $\mbS_n$ free in the variety $\mc{CBM}$. The simplex $\mbS_n$ is a dyadic polytope, and a subreduct of the real simplex $\mbS_n^{r}$ with the same vertices.
The proper faces of $\mbS_n$ are polytopes of smaller dimension.
While each algebraic dyadic simplex is isomorphic to some $\mbS_n$,
not all of them are geometric. An example is given by the dyadic simplex generated by $e_0 = (0, \dots, 0), e'_1 = (3, 0, \dots, 0), \dots, e'_n = (0, \dots, 0, 3)$.
In fact, any affinely independent set of points of $\mbD^n$ generates an algebraic dyadic simplex. We will use the name \emph{algebraic dyadic simplex} for a groupoid isomorphic to some $\mbSn$, while reserving the name \emph{dyadic simplex} for geometric dyadic simplices.

\begin{lemma}\cite[L.~2.2]{CR13}\label{L:subsimpl}
Each $n$-dimensional algebraic dyadic convex set, where $n = 1, 2, \dots$, contains an $n$-dimensional subgroupoid isomorphic to the dyadic simplex $\mbSn$.
\end{lemma}

\begin{corollary}\label{C:subsimpl}
Each $n$-dimensional geometric dyadic convex set $C$, where $n = 1, 2, \dots$, contains infinitely many geometric subgroupoids isomorphic to the dyadic simplex $\mbS_n$.
\end{corollary}
\begin{proof}
First note that, for each $k \in \mb{Z}$, the elements $e'_{i} = (1/2^{k}) e_i$, where $i = 0,1, \dots, n$, generate a geometric dyadic simplex $S'_n$. For any interior point $a$ of $C$, there is $k \in \mbZ$ such that all $a'_i = a + e'_i$ belong to $C$ and generate a geometric dyadic simplex, the translate of $S'_{n}$.
\end{proof}

For a positive integer $n$, let $C$ be an $n$-dimensional algebraic dyadic convex set. Take $k > n$. We say that $C$ is \emph{$k$-generated} if $C$ is minimally generated by a set of $k$ generators. Thus $C$ has a $k$-element set of generators, but not a smaller one.

\begin{lemma}\label{L:simp}
Let $C$ be an $n$-dimensional algebraic dyadic convex set, for $0 < n \in \mathbb Z$. Then $C$ is isomorphic to $\mbSn$ if and only if it is $(n+1)$-generated, or equivalently, it is generated by its $n+1$ vertices.
\end{lemma}
\begin{proof}
We already know that if $C$ is isomorphic to a simplex $\mbSn$, then it is $(n+1)$-generated and that its vertices form a minimal set of generators.

We will show that each $n$-dimensional $(n+1)$-generated algebraic dyadic convex set $C$ is isomorphic to $\mbSn$. This is certainly true for $n = 1$: a one-dimensional $2$-generated convex set is isomorphic to $\mbD_1 = \mbS_1$. Let $n \geq 1$. Consider an $n$-dimensional convex set $C$ minimally generated by the set $X = \{x_0, \dots, x_n\}$. Note that $X$ is affinely independent. (Otherwise $\dim(C)$ would be less than $n$.) The generators of $C$ generate a real $n$-dimensional $(n+1)$-generated polytope, which is obviously a real simplex. Since $\mbSn^r$ is a free barycentric algebra, the mapping
\[
e_0 \mapsto x_0,\, e_1 \mapsto x_1, \dots, e_n \mapsto x_n
\]
extends uniquely to the isomorphism between the barycentric algebras $\mbSn^r$ and $\mr{conv}_{\mbR}(C)$. This isomorphisms restricts to the isomorphism between the dyadic convex sets $\mbSn$ and $C$.
\end{proof}

Note that the convex sets of Lemma~\ref{L:simp} isomorphic to $\mbSn$  are not necessarily geometric.
On the other hand, if $1<i\in\mathbb Z$, then $n$-dimensional and $(n+i)$-generated convex subsets of $\mathbb{D}^n$ are not necessarily isomorphic. Examples are given by the dyadic $3$-generated intervals, and by the (geometric) triangles on $\mbD^2$ with vertices $(-1,0), (0,1)$ and $(0,2k+1)$ for positive integers $k$, which are all generated by four (but not fewer) elements, and no two of which are isomorphic \cite[\S4]{MRS11}.

\begin{corollary}\label{C:geomsiml}
For $0<n\in\mathbb Z$, an $n$-dimensional and $(n+1)$-generated algebraic dyadic convex set $C$ is a geometric simplex precisely if $C$ coincides with its convex $\mbD$-hull $\mr{conv}_{\mbD}(C)$.
\end{corollary}

\subsection{Further comments}\label{S:2.5}

In this paper, dyadic polytopes and their subgroupoids are considered as subreducts of their affine $\mbD$-hulls. After introducing coordinate axes in such a space, a given polytope will be located in the corresponding $\mbD$-module by providing the coordinates of its vertices. An $n$-dimensional affine $\mbD$-space may be considered as a subreduct of the $n$-dimensional real space $\mbR^n$. The free generators of the dyadic simplex $\mbSn$ given by~\eqref{E:freegens} are also free generators of the free affine $\mbD$-space $\mbD^n$ and of the free affine $\mbR$-space $\mbR^n$.

Finally, recall that automorphisms of the affine $\mbD$-space $\mbD^n$ form the $n$-dimensional affine group $\mathrm{GA}(n,\mbD)$ over the ring $\mbD$, the group generated by the linear group $\mathrm{GL}(n,\mbD)$
and the group of translations of the space $\mbD^n$. Moreover, each element of the affine group $\mathrm{GA}(n,\mbD)$ is also an automorphism of the groupoid $(\mbD^n,\circ)$, and transforms any polytope in $\mbD^n$ onto an isomorphic polytope. On the other hand, each automorphism of a nontrivial $n$-dimensional polytope extends to an automorphism of the affine $\mb{D}$-space $\mb{D}^n$. Indeed, by Corollary~\ref{C:subsimpl}, any such polytope contains the free simplex $\mbSn$, with vertices generating the affine hull of the simplex and the polytope.
Isomorphisms of such polytopes in $\mbD^n$ will be considered as restrictions of automorphisms of the affine space $\mbD^n$.

\section{Generating dyadic polytopes}

It was shown in \cite{MRS11} that all dyadic intervals, all dyadic triangles, and more generally all dyadic polygons, are finitely generated. In this section we will show that the same holds for all dyadic polytopes.

Recall that dyadic polytopes, and in particular dyadic intervals, are geometric.
The following basic lemma is a direct consequence of the results of \cite[\S3]{MMR19}.

\begin{lemma}\label{L:3gen}
Let $a, b, c \in \mbD$ with $a < b < c$. If one of the dyadic intervals $[a,b]$ or $[b,c]$ forms a dyadic simplex, then the groupoid generated by $\{a, b, c\}$ coincides with the dyadic interval $[a,c]$.
\end{lemma}
\begin{proof}
If the interval $[a,b]$ is a dyadic simplex, then it is isomorphic to the dyadic interval $[0,1]$. Let $\iota\colon [a,b] \rightarrow [0,1]$ be the required isomorphism assigning $0$ to $a$ and $1$ to $b$. The mapping $\iota$ extends uniquely to the affine space isomorphism $\bar{\iota}\colon \mbD \rightarrow \mbD$ which assigns
$$
d =
\frac{c-a}{b-a}
=\frac{2k+1}{2^n}\in\mathbb D
$$
to $c$. The image $\bar{\iota}([a,c])$ is generated as a groupoid by $\{0, 1, d = \frac{2k+1}{2^n}\}$. Let $h$ be the isomorphism between $\bar{\iota}([a,c])$ and the subgroupoid $H$ generated by $\{0, 2^n, 2k+1\}$. By \cite[Cor.~3.9]{MMR19}, the groupoid $H$ coincides with the interval $[0, 2k+1]$.
Now consider the isomorphism $\overline{i} h$ between $[a,c]$ and $H$. Then note that $(\overline{i} h)^{-1}$ takes $2k+1$ to $c$, $2^n$ to $b$ and $0$ to $a$. Hence $[a,c] = \langle\{a, b, c\}\rangle$.

If $[b,c]$ is a simplex, then we take the isomorphism $\iota\colon [b,c] \rightarrow [0,1]$ assigning $0$ to $c$ and $1$ to $b$. The remaining part of the proof goes similarly to the proof of the first part.
\end{proof}

\begin{lemma}[The Density Lemma]\label{L:points}
Let $A\ne B \in \mbD^n \subset \mbR^n$. Let $l(A,B)$ be the line in the affine $\mbR^n$-space $\mbR^n$ going through $A$ and $B$. Then for any two points $C_1$ and $C_2$ of $l(A,B)$, there is a point $D$ of $l(A,B)$ belonging to $\mbD^n$ and located between $C_1$ and $C_2$.
\end{lemma}
\begin{proof} Let $A = (a_1, \ldots, a_n)$ and $B = (b_1, \ldots, b_n)$. The line $l(A,B)$ is described by the equations
\[
x_i=(b_i-a_i)t+a_i
\]
for $i = 1, \dots, n$. Let $C_1$ and $C_2$ be any points of $l(A,B)$, $C_1$ with coordinates $(b_i-a_i)t_1 + a_i$ and $C_2$ with coordinates $(b_i-a_i)t_2 + a_i$ for some $t_1, t_2 \in \mbR$. There exists a dyadic number $d$ lying between $t_1$ and $t_2$. Then the point $D$ with coordinates $(b_i-a_i)d + a_i$ belongs to $\mbD^n$ and is located between $C_1$ and $C_2$.
\end{proof}

For an illustration of the following lemma, see Figure~\ref{F:1}.

\begin{lemma}\label{L:lines}
Let $P$ be an $n$-dimensional dyadic polytope, and let $\mbS$ be an $n$-dimensional (geometric) simplex contained in $P$. Then for any interior point $A$ of $P$ not contained in $\mbS$,
there exists a point $B$ in a maximal wall of $P$ and a point $C$ in $\rm{int}(\mbS)$ such that $A$ belongs to the intersection $l(B,C) \cap P$ of the (real) line $l(B,C)$ through $B$ and $C$ and $P$.
\end{lemma}

\begin{figure}[bht]
	\begin{center}
		\begin{picture}(340,320)(0,0)

		\put(30,20){\line(1,0){190}}
		\put(30,20){\line(-1,5){30}}
		\put(0,170){\line(1,1){120}}
		\put(120,290){\line(3,-2){222}}
		\put(220,20){\line(1,1){122}}
		\put(30,30){\line(1,0){150}}
		\put(30,30){\line(0,1){150}}
		\put(180,30){\line(-1,1){150}}
		\put( 160,80){\circle*{5}}
		\put(160,85){$A$}
		\put(280,80){\circle*{5}}
		\put(275,67){$B$}
	  	\put(50,80){\circle*{5}}
		\put(42,84){$C$}
			\put(35,80){\circle*{5}}
		\put(35,69){$R_3$}
				\put(95,80){\circle*{5}}
		\put(85,69){$R_4$}
		\put(30,80){\line(1,0){300}}
		\put(220,20){\circle*{5}}
		\put(227,17){$W_1$}
		\put(50,53){\circle*{5}}
		\put(41,42){$C_2$}
		\put(320,120){\circle*{5}}
		\put(320,120){\line(-4,-1){270}}
	    \put(244,35){$R_1$}	\put(250,50){\circle*{5}}
	    \put(250,50){\line(-3,1){200}}
	    \put(325,118){$R_2$}
	
	    		 \put(43,122){$C_1$}	\put(50,117){\circle*{5}}

		\end{picture}
	\end{center}
	\caption{}
	\label{F:1}
\end{figure}

\begin{proof}
We consider $P$ as a subset of $\mbD^n$ and $\mbR^n$ as explained earlier. Let $Q = \mr{conv}_{\mbR}(P)$.
Take $C_1\in \mr{int}(\mbS)$. There is a maximal wall $W$ of $Q$ such that the real line $l(A, C_1)$ through $A$ and $C_1$ has just one point of intersection $R_1$ with $W$. Note that $R_1$ is not necessarily a member of $\mbD^n$. The wall $W$ generates the hyperplane $H = \mr{aff}_{\mbR}(W)$ in $\mbR^n$. Let $W_1 = (d_1, \dots, d_n) \neq R_1$ be one of the vertices of $W$. Since the vertices of $W$ belong to $\mbD^n$, there is a vector $v = (v_1, \dots, v_n)$ with dyadic coordinates such that the hyperplane $H$ is described by the equation
\[
\sum_{i=1}^{n} v_i (x_i - d_i) = 0.
\]
The point $R_1$ is the intersection of the hyperplane $H$ and the line $l(A, C_1)$ given by the equations
\[
x_i = (c_i-a_i)t+a_i,
\]
where $i = 1, \dots, n$, $A = (a_1, \dots, a_n)$ and $C_1 = (c_1, \dots, c_n)$.	
Hence $R_1$ has rational coordinates  ${r_i} / {s_i}$, as the solutions of a system of linear equations with dyadic and hence rational coefficients.
	
The line $l(W_1, R_1)$ has the equations
\[
x_i = (\frac{r_i}{s_i} - d_i)t + d_i,
\]
and contains the dyadic point $D \neq W_1$ obtained for $t$ equal to the least common multiple of all the $s_i$.

Recall that $C_1$ belongs to $\mr{int}(\mbS)$ and $\mbS$ is contained in the real simplex $\mbS^r = \mr{conv}_{\mbR}(\mbS)$.
Choose a point $R_2$ of $l(W_1, R_1) = l(W_1, D) $ different from $R_1$ such that the line $l(R_2, A)$ has a nonempty intersection $C_2$ with $\mr{int}(\mbS^r)$.
Since the line $l(W_1, R_1)$ contains two dyadic points, it follows by the Density Lemma~\ref{L:points}, that there is a dyadic point $B$ located between $R_1$ and $R_2$. The (real) line $l(A,B)$ through $A$ and $B$ has a non-empty intersection with $\mr{int}(\mbS^r)$. Let $R_3$ and $R_4$ be any two different real points in $l(A,B) \cap \mr{int}(\mbS^r)$. Again by Density Lemma, there is a dyadic point $C$ located between $R_3$ and $R_4$.
\end{proof}

\begin{theorem}\label{P:fingen}
Each $n$-dimensional dyadic polytope is finitely generated.
\end{theorem}
\begin{proof}
The proof is by induction on $n$. Suppose that $P$ is a (nontrivial) $n$-dimensional dyadic polytope. If $n = 1$, then $P$ is a dyadic interval, and by the results of \cite{MRS11}, it is (minimally) generated by $2$ or $3$ generators.

Now let $n > 1$. Assume that the theorem holds for all polytopes of dimension smaller than $n$. Let the dimension of $P$ be $n$. The polytope $P$ contains finitely many $(n-1)$-dimensional walls, say $W_1, \dots, W_r$. By the induction hypothesis, each of them is finitely generated. Moreover, by Corollary~\ref{C:subsimpl}, the polytope $P$ contains infinitely many geometric $n$-dimensional simplices. Let $\mbS$ be one of them. Obviously, each point of $\mbS$ is generated by the vertices of $\mbS$.

We will show that any point $A$ of $P$, contained neither in any of the walls $W_1, \dots, W_r$, nor in the simplex $\mbS$, is generated by the generators of these walls and the generators of the simplex.
By Lemma~\ref{L:lines}, there is a point $B$ in a maximal wall $W_i$ and a point $C$ in the interior $\rm{int}(\mbS)$ of $\mbS$ such that the interval with ends $B$ and $C$ contains $A$. Without loss of generality, assume that this interval is $[B, C]$.
Consider the (geometric) subset $\mbS \cap [B,C]$ of the line $l(A,B)$. By Corollary~\ref{C:subsimpl}, this subset contains a one-dimensional simplex, say $[D,C]$, with $D$ between $A$ and $C$. Note that the points $D$ and $C$ are generated by the generators of the simplex $\mbS$, and the point $B$ is generated by the generators of $W_i$. By Lemma~\ref{L:3gen}, the points $B, C$ and $D$ generate the whole interval $[B,C]$ containing the point $A$. It follows that the polytope $P$ is generated by the finite number of the generators of $W_1, \dots, W_r$ and of the simplex~$\mbS$.
\end{proof}

\begin{remark}\label{R:genpol}
Recall from Section~\ref{S:2} that each $n$-dimensional polytope $P$ is the union of its maximal walls and the interior $\rm{int}(P)$ of $P$. The proof of Theorem~\ref{P:fingen} shows that $P$ is generated by the generators of these walls and $n+1$ vertices of an $n$-dimensional geometric simplex $\mbS$ contained in $P$. As the maximal walls are $(n-1)$-dimensional polytopes themselves, the proof provides a recursive method of finding generators of $P$. When the generators of the maximal walls have already been determined, it suffices to take the generators of the maximal walls and the vertices of $\mbS$ as the generators of $P$.
\end{remark}

As was shown in \cite[\S3]{MMR19}, a finite number of distinct elements of the dyadic line $\mbD$ generate a convex subset isomorphic to a dyadic interval, a one-dimensional dyadic polytope. However, the relation between the finitely generated subgroupoids of the groupoid $(\mbD^n, \circ)$ and the $n$-dimensional polytopes for $n > 1$ has not yet been established. We discuss this relationship in the following section.

\section{Finitely generated subgroupoids of $\mbD^n$}

We already know that each finitely generated subgroupoid of $\mbD$ is isomorphic to a dyadic interval.
We would like to know if every finitely generated subgroupoid of $\mbD^n$ is isomorphic to a dyadic polytope.
The following example shows that if $n > 1$, then the situation may be more complicated.

\begin{example}\label{Ex:notdpol}
Consider the following four elements of the plane $\mbD^2$:
\[
A_0 = (0,0),\, A_1 = (1,3),\, A_2 = (3,0),\, A_3 = (1,1).
\]
(See Figure~\ref{F:2}.)
\begin{figure}[bht]
	\begin{center}
		\begin{picture}(170,160)(0,0)

		\put(0,20){\vector(1,0){160}}
		\put(20,0){\vector(0,1){150}}

		\put(18,140){\line(1,0){4}}
		\put(18,60){\line(1,0){4}}
		\put(10,57){$1$}
		\put(60,18){\line(0,1){4}}
		\put(57,8){$1$}
		\put(10,138){$3$}
			\put(140,20){\circle*{5}}
		\put(60,60){\circle*{5}}
\put(65,65){$A_3$}
				\put(60,140){\circle*{5}}
				\put(142,24){$A_2$}
		\put(137,8){$3$}

		\put(20,20){\circle*{5}}
		\put(5,7){$A_0$}
	
		\put(60,144){$A_1$}

		\thicklines

		\put(20,20){\line(1,0){120}}
		\put(20,20){\line(1,1){40}}
		\put(20,20){\line(1,3){40}}
		\put(140,20){\line(-2,3){80}}
 	\put(140,20){\line(-2,1){80}}
		
		\end{picture}
	\end{center}
	\caption{}
	\label{F:2}
\end{figure}
\noindent We are looking for the groupoid $G$ generated by these elements. First note that each triple of them generates an (algebraic) dyadic simplex. However two of these simplices, the one generated by $A_0, A_1, A_2$ and the one generated by $A_0, A_2, A_3$, do not contain all the points of the interval $[A_0, A_2]$, e.g. the points $(1,0)$ and $(2,0)$. It follows that these simplices are not geometric. The other two triples  generate geometric dyadic simplices. This can be shown by providing automorphisms of the dyadic plane $\mbD^2$ taking three of the corresponding generators to the points $(0,0), (1,0), (0,1)$, which are known to be the free generators of the dyadic simplex $\mbS_2$. For example, (right) multiplication by the dyadic matrix
\renewcommand{\arraystretch}{1.5}
\[
M_1 =
    \begin{bmatrix}
      -\tfrac{1}{2} & \tfrac{3}{2} \\
      \tfrac{1}{2} & -\tfrac{1}{2} \\
    \end{bmatrix}
\,,
\]
with determinant equal to $-1/2$, provides an automorphism of the affine $\mbD$-space $\mbD^2$ which takes $A_0$ to $(0,0)$, $A_1$ to $(1,0)$ and $A_3$ to $(0,1)$. Similarly, translation by $(-1,-1)$ followed by the automorphism given by the matrix
\[
M_2 = \begin{bmatrix}
      \frac{1}{4} & \frac{1}{2} \\
      \frac{1}{2} & 0 \\
    \end{bmatrix}\,,
\]
with determinant equal to $-1/4$,
takes the points $A_1$, $A_2$, $A_3$ to the points $(0,0), (1,0), (0,1)$.
In particular, the points $A_0, A_1, A_2, A_3$ generate all the points of the (closed) dyadic triangles $A_0 A_1 A_3$ and $A_1 A_2 A_3$.
\renewcommand{\arraystretch}{1}

We will show that the (dyadic) points of the interior of the triangle $A_0 A_2 A_3$ are generated by the points $A_0, A_1, A_2, A_3$. Let $B$ be any point of the interior of the triangle $A_0 A_2 A_3$. The dyadic line $l(A_0,B)$ through the points $A_0$ and $B$ has a nonempty intersection with the triangle $A_1 A_2 A_3$. By Corollary~\ref{C:subsimpl}, this intersection contains a closed interval $[C,D]$, with $C$ between $B$ and $D$, which is a  (geometric) dyadic simplex. By Lemma~\ref{L:3gen}, the points $A_0, C, D$ generate the dyadic interval $[A_0,D]$, containing the point $B$, hence also generated by the points $A_0, C, D$.

Finally, recall that not all the points of the interval $[A_0, A_2]$ are generated by $A_0$ and $A_2$. It follows that the groupoid $G$ consists of the whole interior of the triangle $A_0 A_1 A_2$, the dyadic intervals $[A_0, A_1]$ and $[A_1, A_2]$, and the proper subgroupoid of the interval $[A_0, A_2]$ generated by its ends, and consisting of the points $(d,0)$ with $0 \leq d \leq3$ with $d$ divisible by $3$. Hence the groupoid $G$ is not a dyadic triangle.
\end{example}

The groupoid $G$ of Example~\ref{Ex:notdpol} is a finitely generated subgroupoid of $(\mbD^2, \circ)$, but it is not a dyadic polytope. Furthermore, the groupoid $G$ is not isomorphic to any dyadic polytope. To see this, observe first that since $\mr{int}(A_0 A_1 A_2) \cup ]A_0 A_1] \cup [A_1 A_2[$ is geometric, it follows that the affine $\mbD$-hull of $G$ is the whole plane $\mbD^2$, and its convex $\mbD$-hull $\rm{conv}_{\mbD}(G)$ is the (closed) dyadic triangle with vertices $A_0, A_1, A_2$. Recall from Section~\ref{S:2} that an algebraic convex set $C$ is geometric precisely when it coincides with its convex $\mbD$-hull. Clearly, $G$ and $\rm{conv}_{\mbD}(G)$ differ.

We will show, however, that finitely generated finite-dimensional algebraic dyadic convex sets are not far from being polytopes.

\begin{definition}\label{D:qpol}
A subgroupoid $D$ of a dyadic polytope $C$ is called a \emph{semipolytope} if $C$ and $D$ have the same vertices and $D \, \cap \, \rm{int} (C) = \rm{int}(C)$.
\end{definition}

In particular, $C = \rm{conv}_{\mbD}(D)$. Thus the interiors of $C$ and $D$ coincide. Obviously, a semipolytope is an algebraic dyadic convex set.

\begin{remark}
In \cite{MR16}, the concept of a \emph{quasipolytope} in an affine $\mbR$-space was introduced. This is a convex set $A$ whose closure $\overline{A}$ in the affine $\mbR$-space it generates is a polytope. Note that $A$ does not need to contain the vertices of $\overline{A}$. Now define a real semipolytope, similarly as in the dyadic case. Such real semipolytopes are simply polytopes. On the other hand, quasipolytopes may be defined in the dyadic case as well. (We consider the usual Euclidean topology on $\mbR^n$, and then restrict it to $\mbD^n$.) Each dyadic semipolytope is a quasipolytope. However, dyadic quasipolytopes are not necessarily semipolytopes. Examples are provided by open polytopes (interiors of polytopes). Note also that the topological closures of geometric dyadic convex sets admit a purely algebraic description \cite{CR13}.
\end{remark}

\begin{example}
The groupoid $G$ of Example~\ref{Ex:notdpol} is a semipolytope. It has the same interior as the (closed) triangle $A_0 A_1 A_2$ and the same vertices.
\end{example}

\begin{example}\label{Ex:qpol}
Consider the semipolytope $S$ in the plane $\mbD^2$ consisting of the interior of the triangle with vertices $A_0 = (0,0)$, $A_1 = (0,1)$ and $A_2 = (9,0)$, the sides $A_0 A_1$ and $A_1 A_2$, and the subgroupoid $G$ of the side $A_0 A_2$ consisting of the pairs $(d,0)$ with $d$ divisible by $3$ (and generated by the points $A_0, A_2$ and $D = (3,0)$). The convex $\mbD$-hull of $S$ is the (closed) triangle $A_0 A_1 A_2$. Choose the point $E = (1, 1/2)$ of $S$. Since the points $A_0, E, A_1$ generate the triangle $A_0 E A_1$, Corollary~\ref{C:geomsiml} implies that $A_0 E A_1$ is geometric. Using this fact, one may show (as in Example~\ref{Ex:notdpol}) that each interior point of $S$ is generated by the points $A_0, A_1, A_2$ and $E$. Hence $S$ is finitely generated by five points $A_0, A_1, A_2, D, E$. (See Figure~\ref{F:3}.) Note that $G$ forms a wall of $S$, which is not a polytope in $\mbD^2$. Neither is it a polytope in the $x$-axis, which is a subspace of $\mbD^2$. However, $G$ is a polytope in its affine $\mbD$-hull, which consists of the elements $(d,0)$ of the $x$-axis with $d$ divisible by $3$.

\begin{figure}[bht]
	\begin{center}
		\begin{picture}(240,160)(0,0)

		\put(0,20){\vector(1,0){230}}
		\put(20,0){\vector(0,1){140}}
		
		\put(20,110){\circle*{5}}
		\put(40,18){\line(0,1){4}}
		\put(37,8){$1$}
		\put(10,118){$1$}
		\put(22,118){$A_1$}
		\put(200,20){\circle*{5}}

		\put(202,24){$A_2$}
		\put(197,8){$9$}

		\put(80,20){\circle*{5}}
		\put(82,24){$D$}
		\put(77,8){$3$}
		
		\put(20,20){\circle*{5}}
		\put(5,7){$A_0$}
		
		\put(40,65){\circle*{5}}
		\put(43,65){$E$}

		\thicklines
		
		\put(20,110){\line(2,-1){180}}
		\put(20,20){\line(2,5){18}}
		\put(20,110){\line(2,-5){18}}
		
		\end{picture}
	\end{center}
	\caption{}
	\label{F:3}
\end{figure}

\end{example}

\begin{example}
Consider again the semipolytope $S$ of Example~\ref{Ex:qpol} and the dyadic triangle $A_0 A_1 A_2$ with vertices $A_0, A_1, A_2$.
Note that the matrix
\[
\begin{bmatrix}
5 & 0 \\
0 & 3
\end{bmatrix}
\]
provides an isomorphism $\kappa$ from $\mbD^2$ with the dyadic triangle $A_0 A_1 A_2$ as its subgroupoid onto the subspace $$\{(5d,0) \mid d \in \mbD\} \times \{(0, 3d) \mid d \in \mbD\} = 5\mbD \times 3\mbD$$ of $\mbD^2$. Then note that $5\mbD \times 3\mbD$ is the affine $\mbD$-hull of $\kappa(A_0 A_1 A_2)$. The image $\kappa(A_0 A_1 A_2)$ is an algebraic dyadic convex subset of the affine $\mbD$-space $\mbD^2$ and isomorphic to the dyadic triangle $A_0 A_1 A_2$. The image $\kappa(S)$ is a finitely generated algebraic dyadic convex subset of $\mbD^2$ isomorphic to the semipolytope $S$, but is not a semipolytope.
\end{example}

Note that the walls of the semipolytopes in $\mbD^2$ of Examples~\ref{Ex:notdpol} and ~\ref{Ex:qpol} are also semipolytopes; not necessarily in $\mbD^2$, but relative to their affine $\mbD$-hulls. We will see that this is a property of general semipolytopes. First note that if $S$ is a semipolytope in $\mbD$, then it is an interval, and hence a polytope.  Note however that a polytope in $m\mbD$ for odd $m \neq 1$ is not a polytope in $\mbD$. A two-dimensional semipolytope $S$ in $\mbD^2$ has the same interior as its convex $\mbD$-hull $\mr{conv}_{\mbD}(S)$.
The walls of $S$ are subgroupoids of some dyadic lines in $\mbD^2$, but they are not necessarily geometric.
More generally, $0$-dimensional walls of an $n$-dimensional semipolytope $S$ are the vertices of $S$ and $1$-dimensional walls are the intersections of $S$ with the edges of $\mr{conv}_{\mbD}(S)$.
An $n$-dimensional semipolytope $S$ has the same interior as its convex $\mbD$-hull $\mr{conv}_{\mbD}(S)$, and this interior is obviously geometric. However, the walls of $S$ are subgroupoids of some affine subspaces of $\mbD^n$, and do not need to be geometric.

Let $C$ be an algebraic dyadic convex subset of an $n$-dimensional subspace $A$ of $\mbD^n$.
Then the \emph{convex $A$-hull} $\mr{conv}_{A}(C)$ of $C$ in $A$ is the intersection of $\mr{conv}_{\mbR}(C)$ with $A$,
\[
\mr{conv}_{A}(C) = \mr{conv}_{\mbR}(C) \cap A.
\]
Since $\mr{aff}_{\mbD}(C)$ is a subspace of $\mbD^n$, it follows that
\begin{align}\label{E:Aconv}
\mr{conv}_{\mr{aff}_{\mbD}(C)}(C) &= \mr{conv}_{\mbR}(C) \cap \mr{aff}_{\mbD}(C)\\ &= \mr{conv}_{\mbD}(C) \cap \mr{aff}_{\mbD}(C).\notag
\end{align}

\begin{proposition}\label{P:wallsepol}
Let $S$ be an $n$-dimensional semipolytope. Then the following hold.
\begin{itemize}
\item[(a)] A subset $W$ of $S$ is a wall of $S$ if and only if it is the intersection with $S$ of a wall of the convex $\mbD$-hull $\mr{conv}_{\mbD}(S)$ of $S$.
\item[(b)] Each wall $W$ of $S$ is a semipolytope in the affine $\mbD$-hull
$\mr{aff}_{\mbD}(W)$ of $W$.
\end{itemize}
\end{proposition}
\begin{proof}
$(\rm a)$
First note that $S$ and $\mr{conv}_{\mbD}(S)$ have the same vertices and the same interior, and by the last remarks of Section $2.3$, the walls of
$\mr{conv}_{\mbD}(S)$ are determined by the vertices they contain. If $W'$ is a wall of $\mr{conv}_{\mbD}(S)$, then $W' \cap S$ is a wall of $S$. On the other hand,
if $W$ is a wall of $S$, then $W = S \cap F$, where $F$ is the wall of $\mr{conv}_{\mbD}(S)$ determined by the vertices of $W$.

$(\rm b)$ Let $A := \mr{aff}_{\mbD}(W)$ be the affine $\mbD$-hull of $W$.  Note that, by \eqref{E:Aconv}, $\mr{conv}_{A}(W) = \mr{conv}_{\mbD}(W) \cap A$. Hence $\mr{int}(\mr{conv}_{A}(W)) = \mr{int}(W)$, and obviously $\mr{conv}_{A}(W)$ and $W$ have the same vertices.
\end{proof}

\begin{definition}
Consider an $n$-dimensional subspace $A$ of $\mbD^n$.
\begin{enumerate}
\item[(a)]
A \emph{dyadic $n$-dimensional polytope in $A$} or \emph{$A$-polytope} is the intersection with the space $A$ of an $n$-dimensional real polytope whose vertices lie in the space $A$.
\item[(b)]
A subgroupoid $D$ of a dyadic polytope $C$ in $A$ is called a \emph{semipolytope in $A$} or an \emph{$A$-semipolytope} if $C$ and $D$ have the same vertices and the same (relative) interiors. Thus a semipolytope is a semipolytope in $\mbD^n$.
\end{enumerate}
\end{definition}

In what follows, we are mostly interested in finitely generated algebraic dyadic convex sets. Note that such sets are finite-dimensional, since their affine $\mb{D}$-hulls are finitely generated.

\begin{lemma}\label{L:fingenalg}
Let $C$ be a finitely generated algebraic dyadic convex set with the $n$-dimensional affine $\mbD$-hull $A$, a subspace of $\mbD^n$.
Then the following hold.
\begin{itemize}
\item[(a)] The walls of $C$ are the intersections with $C$ of the walls of $\mr{conv}_{A}(C)$.
\item[(b)] The union of the maximal walls of $C$ generates a subgroupoid of the groupoid $C$.
\item[(c)] The equality $\rm{int}(C) = \rm{int}(\mr{conv}_{A}(C))$ holds.
\item[(d)] The dyadic convex sets $C$, $\mr{conv}_{A}(C)$, $\mr{conv}_{\mbD}(C)$ and the real convex set $\mr{conv}_{\mbR}(C)$ share the same vertices.
\item[(e)] The convex $A$-hull $\mr{conv}_{A}(C)$ of $C$ is an $A$-polytope, and hence its walls are also $A$-polytopes.
\item[(f)] The dyadic convex set $C$ is a semipolytope in $A$.
\end{itemize}
\end{lemma}
\begin{proof}
$(\rm a)$ and $(\rm b)$ are obvious, while $(\mr c)$ follows directly from the fact that $A = \rm{aff}_{\mbD}(C)$.

$(\rm d)$ Since $C$ is finitely generated, it follows that the barycentric algebra $\mr{conv}_{\mbR}(C)$ is also finitely generated. Hence it is a real polytope. Its vertices belong to $C$, to $\mr{conv}_{A}(C)$, and to $\mr{conv}_{\mbD}(C)$.

$(\rm e)$ Since $\mr{conv}_{\mbR}(C)$ is a real polytope, it follows that $\mr{conv}_{\mbD}(C) = \mr{conv}_{\mbR}(C) \cap \mbD^n$ is a dyadic polytope. Then, by~\eqref{E:Aconv} and $(\rm d)$, $\mr{conv}_{A}(C)$ is a polytope in $A$.

$(\rm f)$ This follows directly by $(\rm c), (\rm d)$ and $(\rm e)$.
\end{proof}

\begin{proposition}\label{P:fingencon}
Each finitely generated algebraic dyadic convex set is isomorphic to a semipolytope.
\end{proposition}
\begin{proof}
Let $C$ be an algebraic convex subset of $\mbD^p$ generated by a finite set $X = \{x_0, x_1, \cdots, x_r\}$. Let $A = \mr{aff}_{\mbD}(C)$ with $\dim A = n \leq p$ and $n \leq r$.

By Lemma~\ref{L:fingenalg}, $C$ is a semipolytope in $A$.
Moreover, $C$, $\mr{conv}_{A}(C)$, $\mr{conv}_{\mbD}(C)$ and $\mr{conv}_{\mbR}(C)$ have the same set $V$ of vertices contained in $X$. Moreover, $\rm{int}(C) = \rm{int}(\mr{conv}_{A}(C))$.

By Lemma~\ref{L:subsimpl}, $C$ contains an $n$-dimensional algebraic simplex  $\mbS$ isomorphic as a groupoid to the (geometric) dyadic simplex $\mbS_n$. The simplex $\mbS$ is also contained in $A = \rm{aff}_{\mbD}(C)$, and its vertices $s_0, \dots, s_n$ freely generate the affine $\mbD$-space $A$, which is a subspace of $\mbD^n$ (isomorphic to $\mbD^n$) and the affine $\mbR$-space $\mbR^n$.

On the other hand, the space $\mbD^n$ contains the dyadic simplex $\mbS_n$. Its vertices $e_0, \dots, e_n$ freely generate the groupoid $\mbS_n$, the affine $\mbD$-space $\mbD^n$, and the affine $\mbR$-space $\mbR^n$.

The bijective mapping $\iota \colon s_i \mapsto e_i$ assigning the vertices of $\mbS$ to the vertices
of $\mbS_n$ extends uniquely to the authomorphism $\ov{\ov{\iota}}$ of the affine $\mbR$-space $\mbR^n$, restricting to the isomorphism $\ov{\iota}$ between the affine $\mbD$-spaces $A$ and $\mbD^n$. The authomorphism $\ov{\ov{\iota}}$ restricts to the groupoid isomorphism between $\mr{conv}_{A}(C)$ and its image $\ov{\iota}(\mr{conv}_{A}(C)) = \ov{\ov{\iota}}(\mr{conv}_{A}(C))$ in $\mbD^n$.

Our aim is to prove that $\ov{\iota}(C)$ is a semipolytope, which implies that $C$ is isomorphic to a semipolytope.

First note that $\ov{\iota}(C)$ is a subgroupoid of $\ov{\iota}(\mr{conv}_{A}(C))$.
Since $\ov{\iota}$ is an isomorphism, it follows by Lemma~\ref{L:fingenalg}, (d) and (c), that
$\ov{\iota}(C)$ and $\ov{\iota}(\mr{conv}_A(C))$ have the same vertices and the same (relative) interiors.
It remains to show that $\ov{\iota}(\mr{conv}_{A}(C))$ is a dyadic polytope. By~\eqref{E:pol}, this means that
\begin{equation}\label{E:dpol}
\ov{\iota}\big (\mr{conv}_A(C)\big ) = \mr{conv}_{\mbR}\big (\ov{\iota}(\mr{conv}_A(C)) \big ) \cap \mbD^n.
\end{equation}
The inclusion $\ov{\iota}\big (\mr{conv}_A(C)\big ) \subseteq \mr{conv}_{\mbR}\big (\ov{\iota}(\mr{conv}_A(C)) \big ) \cap \mbD^n$ follows by~\eqref{E:Aconv}. To show that the reverse inclusion holds as well,
first observe that
\begin{equation}\label{E:convconv}
\mr{conv}_{\mbR}\big (\ov{\iota}(\mr{conv}_A(C)) \big ) = \ov{\ov{\iota}} \big (\mr{conv}_{\mbR}(C) \big ).
\end{equation}
Now let $x \in
\mr{conv}_{\mbR}\big (\ov{\iota}(\mr{conv}_A(C)) \big ) \cap \mbD^n = \ov{\ov{\iota}} \big (\mr{conv}_{\mbR}(C) \big ) \cap \mbD^n$.
Then $x = \sum a_i e_i$ for some $a_i\in \mbD$. Hence $\ov{\iota}^{-1}(x) = \sum a_i s_i \in A$ and $\ov{\iota}^{-1}(x) \in \mr{conv}_{\mbR}(C)$. Therefore
$\ov{\iota}^{-1}(x) \in \mr{conv}_{A}(C)$, and consequently, $x \in \ov{\iota}(\mr{conv}_A(C))$. Hence $\ov{\iota}(\mr{conv}_A(C))$ is a dyadic polytope.
\end{proof}

Our aim is to prove that all semipolytopes are finitely generated.
First we will need some additional information concerning real convex polytopes.
Let $C$ be an $(n+1)$-dimensional convex polytope in the affine $\mbR$-space $\mbR^{n+1}$ with vertex set $V = \{x_0, \dots, x_k\}$, where $k > n$,  and with the set $\mr{W}$ of $n$-dimensional walls $W_1, \dots, W_r$. Choose one interior point $y_i$ from each wall $W_i$, and call it an \emph{anchor} on $W_i$. Let $Y$ be the set $\{y_1, \dots, y_r\}$ of all anchors. The anchors $y_i$ generate the subpolytope $P = \rm{conv}_{\mbR}(Y)$ of $C$, which will be called an \emph{inner polytope} of $C$.
(For an example, see Figure~\ref{F:4}.)
Note that an inner polytope has the same dimension as the polytope $C$.

\begin{figure}[bht]
	\begin{center}
		\begin{picture}(170,160)(0,0)

		\put(20,20){\line(1,0){130}}
		\put(20,20){\line(-1,5){15}}
		\put(5,95){\line(1,1){60}}
		\put(65,155){\line(2,-1){146}}
		\put(150,20){\line(1,1){62}}
			\put(20,20){\circle*{5}}
			\put(5,20){$x_0$}
			\put(150,20){\circle*{5}}
		\put(155,20){$x_1$}
			\put(5,95){\circle*{5}}
		\put(-10,100){$x_4$}
			\put(65,155){\circle*{5}}
		\put(65,160){$x_3$}
			\put(212,82){\circle*{5}}
		\put(208,88){$x_2$}
				\put(50,20){\circle*{5}}
		\put(50,10){$y_1$}
				\put(50,20){\line(5,1){123}}
					\put(174,44){\circle*{5}}
		\put(180,40){$y_2$}
					\put(174,44){\line(0,1){55}}
		\put(174,100){\circle*{5}}
		\put(180,105){$y_3$}
				\put(33,123){\line(6,-1){145}}
		\put(33,123){\circle*{5}}
		\put(26,130){$y_4$}
		\put(33,123){\line(-1,-3){22}}
			\put(12,60){\circle*{5}}
		\put(-3,60){$y_5$}
		\put(10,60){\line(1,-1){40}}

		\put(150,20){\circle*{5}}
		\put(155,20){$x_1$}
		\put(5,95){\circle*{5}}
		\put(-10,100){$x_4$}
		\put(65,155){\circle*{5}}
		\put(65,160){$x_3$}
		\put(212,82){\circle*{5}}
		\put(208,88){$x_2$}

	\end{picture}
	\end{center}
	\caption{}
	\label{F:4}
\end{figure}

If $S$ is an $n$-dimensional semipolytope in an $m$-dimensional dyadic space with $n \leq m$, then without loss of generality, we will consider $S$ as a semipolytope in the $n$-dimensional dyadic space that forms the dyadic affine hull of $S$.

Let $S$ be a (nontrivial) $n$-dimensional semipolytope in the space $\mbD^n$. It is a subgroupoid of its convex $\mbD$-hull $S_d := \mr{conv}_{\mbD}(S)$, which is in turn a subgroupoid of its convex $\mbR$-hull $S_r = \mr{conv}_{\mbR}(S)$. Recall that $S$, $S_d$ and $S_r$ have the same vertices. Choose a set of anchors of the maximal walls of the (real) polytope $S_r$, belonging to the semipolytope $S$. The anchors generate the inner polytope $P_r$ of $S_r$. The intersection $P := P_r \cap S$ is a (dyadic) subpolytope of $S$, and is called an \emph{inner polytope} of $S$.
By Proposition~\ref{P:wallsepol}, the walls of $S$ are the intersections of the walls of $S_r$ with $S$, and each one is a semipolytope in its affine $\mbD$-hull.

Now it remains to show that all semipolytopes are finitely generated.
We start with an extension of Lemma~\ref{L:lines}.

\begin{lemma}\label{L:Slines}
Let $S$ be an $n$-dimensional semipolytope, and let $P$ be an inner polytope of $S$. Then for any interior point $A$ of $S$ not contained in $P$,
there exists a point $B$ in a maximal wall of $S$ and a point $C$ in $\rm{int}(P)$ such that $A$ belongs to the intersection $l(B,C) \cap S$ of the (real) line $l(B,C)$ through $B$ and $C$ and $S$.
\end{lemma}
\begin{proof}
The proof is very similar to the proof of Lemma~\ref{L:lines}, with the polytope $P$ replaced by the semipolytope $S$, and the simplex $\mbS$ replaced by the inner polytope $P$ of $S$. The only difference concerns the location of the point $B$, which can be obtained as follows. Let $k(W_1,R_1) = l(W_1,R_1) \cap \rm{aff}_{\mbD}(W)$. Let $\iota \colon k(W_1,R_1) \rightarrow \rm{aff}_{\mbD}(W)$ be an isomorphism of affine $\mbD$-spaces and let $\ov{\iota} \colon l(W_1,R_1) \rightarrow \mbR$ be its extension to the affine $\mbR$-space isomomorhism. By the Density Lemma~\ref{L:lines}, there is a point $B'$ of $\rm{aff}_{\mbD}(W)$ located between $\ov{\iota}(R_1)$ and $\ov{\iota}(R_2)$. Then $B = \iota^{-1}(B')$ is a point of
$\rm{aff}_{\mbD}(W)$ located between $R_1$ and $R_2$.
\end{proof}

\begin{proposition}\label{P:spolfing}
Each semipolytope is finitely generated.
\end{proposition}
\begin{proof}
Let $S$ be a (nontrivial) $n$-dimensional semipolytope in the space $\mbD^n$.
The proof is by induction on $n$.
If the dimension of $S$ is $1$, then its convex $\mbD$-hull $S_d := \mr{conv}_{\mbD}(S)$
is an interval, and by results of \cite{MRS11}, it is (minimally) generated by $2$ or $3$ elements. In this case, the semipolytope $S$ and the polytope $S_d$ coincide.

Now assume that the result holds for all dimensions less or equal to $n$. Consider a semipolytope $S$ of dimension $n+1$ with vertices $\{x_0, x_1, \dots, x_k\}$, where $k > n$, but minimally generated by a possibly bigger set of generators containing the set of all vertices. The walls of $S$ are intersections of the walls of the polytope $S_d$ with the semipolytope $S$, and $S$ has a finite number of walls. By Proposition~\ref{P:wallsepol}, all proper nontrivial walls of $S$ (each of dimension not greater than $n$) are semipolytopes in their affine $\mbD$-hulls. By the induction hypothesis, they are finitely generated.

Now choose (dyadic) anchors $y_1, \dots, y_r$ of the maximal walls\\ $W_1, \dots, W_r$ of $S$, and consider the corresponding inner polytope $P$ of $S$.
Recall that, by the induction hypothesis, all maximal proper walls are finitely generated, and by Theorem~\ref{P:fingen}, $P$ is also finitely generated.
We will show that each point of $S$, not contained in a maximal proper wall $W_i$ or in the inner polytope $P$, is generated by an element of a maximal proper wall and two elements of the inner polytope $P$.
The proof is similar to the last part of the proof of Theorem~\ref{P:fingen}.
Let $A$ be an interior point of $S$ not contained in the inner polytope $P$.
By Lemma~\ref{L:Slines}, there is a point $B$ in a maximal wall $W_i$ of $S$ and a point $C$ in $\rm{int}(P)$ such that $A$ belongs to the intersection $l(B,C) \cap S$ of the (real) line $l(B,C)$ through $B$ and $C$ and $S$. Without loss of generality assume that $A$ belongs to the interval $[B,C]$.
Then, by Corollary~\ref{C:subsimpl}, the intersection $P \cap l(B,C)$ contains a point $D$, with $D$ between $A$ and $C$, such that the interval $[D,C]$ is a one-dimensional (geometric) simplex. By Lemma~\ref{L:3gen}, the points $B, C$, and $D$ generate the whole dyadic interval $[B,C]$ of the line $l(B,C)$ containing the point $A$. Recall that the points $C$ and $D$ are generated by the generators of $P$, and the point $B$ is generated by the generators of the wall $W_{i}$. It follows that the semipolytope $S$ is generated by the finite number of generators of the finite number of maximal proper walls $W_i$ and the finite number of generators of the polytope $P$.
\end{proof}

\section{Main result}

The results obtained in the previous sections may be summarized by the following theorem.

\begin{theorem}\label{T:main}
A subgroupoid $(D, \circ)$ of the groupoid $(\mbD^n, \circ)$ is finitely generated precisely if it is isomorphic to a semipolytope.
\end{theorem}

\begin{remark}\label{R:genspol}
Recall that each $n$-dimensional semipolytope $S$ is the standard sum of an inner polytope $P$ and outer subsemipolytopes $P_i$ determined by $P$.
The proof of Theorem~\ref{T:main} shows that each point of $S$ contained in any of the $P_i$, but not contained in a maximal proper wall or in the inner polytope $P$, is generated by an element of a maximal proper wall of $S$ and two elements of the inner polytope $P$. As the maximal walls and the inner polytope $P$ are all finitely generated, the proof provides a recursive method of finding generators of $S$. To generate the semipolytope $S$, it is enough to take the generators of all the maximal walls, and of the polytope $P$.
\end{remark}

The following example shows that there are geometric dyadic convex sets, with a finite number of vertices, which are not finitely generated.

\begin{example}
Consider the unit disc in the real plane, and its intersection $U$ with the dyadic plane. The boundary of the unit disc is the circle given by the equation $x^2 + y^2 = 1$. Note that the points $(-1,0)$,  $(0,-1)$, $(1,0)$ and $(0,1)$ belong to the circle. For any dyadic point $(j/2^m, k/2^n)$ of the circle, with say $0 < m \leq n$, and odd integers $j$ and $k$, we have
\begin{equation}\label{E:cex}
k^2 = 2^{2(n-m)}(2^{2m} - j^2).
\end{equation}
(The case when $0 < n \leq m$, will be similar.)
Note that the left-hand side of~\eqref{E:cex} is odd. If $m \neq n$, then  the right-hand side of~\eqref{E:cex} is even. If $m = n$, then $j^2 + k^2 = 2^{2n}$. Let $j = 2a + 1$ and $k = 2b + 1$. Then $1/2(j^2 + k^2) = 2(a^2 +a + b^2 + b) + 1 = 2^{2n-1}$. This holds only if $n = 1/2$, again a contradiction. Consequently, the unique dyadic points belonging to the circle are $(-1,0)$,  $(0,-1)$, $(1,0)$ and $(0,1)$. It follows that $U$ consists of the dyadic points of the interior of the real unit disc and these four points. The four points generate the geometric square contained in $U$. But no finite set of elements of $U$ generates the full set of points of $U$ not contained in the square.
\end{example}

\section{Counting the number of generators}

In this concluding section, we will illustrate how to use the main results of the paper to find a minimal set of generators of some semi-polygons, and estimate their number.

In \cite{MRS11}, it was shown that all dyadic polygons (and in particular all intervals and triangles) are finitely generated. However, the methods we used in the proof did not provide an efficient way to find a minimal set of generators, or to calculate their number.
In this section we will show how to use the results of the previous sections to solve such problems. In particular, we will show how to find a minimal set of generators of dyadic triangles and semipolytopes closing to triangles, i.e. semipolytopes with dyadic triangles as their closure. We will call such semipolytopes \emph{semitriangles}.

We already know that finitely generated (algebraic) convex subsets of the dyadic line $\mbD$ are isomorphic to intervals, and that each dyadic interval is minimally generated by $2$ or $3$ elements. As shown in \cite{MRS11},
each nontrivial interval of $\mbD$ is isomorphic to some dyadic interval
$\mbD_k = [0, k]$, where $k$ is an odd positive integer. Two such intervals are isomorphic
precisely when their right hand ends are equal. Any set of generators must contain the vertices. If $k \neq 1$, then to generate $\mbD_k$, one more generator is needed. This can be the number $1$, for example. As semipolytopes in $\mbD$ coincide with polytopes, we do not need to consider them separately.

\subsection{Dyadic triangles and polygons}\label{S:dyadtriangles}

Now let us consider dyadic triangles.
Dyadic triangles were classified in \cite{MRS11}, with some improvements given in \cite{MMR19a}.
First recall the following lemma.

\begin{lemma}\cite[L.~5.2]{MRS11}\label{L:firstquadr}
Each dyadic triangle is isomorphic to a triangle $ABC$ contained in the first quadrant, with one vertex, say $A$, located at the origin, and with the vertices $B$ and $C$ having non-negative integer coordinates. (See Figure~\ref{F:5}.)
\end{lemma}

\begin{figure}[bht]
\begin{center}
\begin{picture}(170,140)(0,0)

\put(0,20){\vector(1,0){160}}
\put(20,0){\vector(0,1){140}}
\put(20,20){\line(1,1){60}}

\put(20,80){\line(1,0){120}}
\put(10,77){$n$}
\put(80,20){\line(0,1){100}}
\put(77,8){$i$}
\put(20,120){\line(1,0){120}}
\put(10,117){$j$}
\put(140,20){\line(0,1){100}}
\put(137,8){$m$}

\put(80,80){\circle*{3}}
\put(82,69){$G$}

\put(20,20){\circle*{5}}
\put(7,7){$A$}
\put(80,120){\circle*{5}}
\put(77,127){$B$}
\put(140,80){\circle*{5}}
\put(145,76){$C$}

\thicklines

\put(20,20){\line(2,1){120}}
\put(20,20){\line(3,5){60}}
\put(140,80){\line(-3,2){60}}

\end{picture}
\end{center}
\caption{}
\label{F:5}
\end{figure}

\noindent The integers $i, j, m, n$ in Figure~\ref{F:5} may be chosen so that $0 \leq i < m$,  $0 \leq  n < j$, and with odd $\gcd\{i,m\}$ and $\gcd\{j,n\}$. Such a triangle will be denoted $T_{i,j,m,n}$.

To provide a more detailed description of triangles $ABC$ located as in Lemma~\ref{L:firstquadr}, we need some further definitions.
An interval of $\mbD$ is of \emph{type $k$} if it is isomorphic to $\mbD_k$. A dyadic triangle has \emph{boundary type} $(r,s,t)$, where $r, s, t$ are odd positive integers, if its sides have respective types $r, s$ and $t$.
A dyadic triangle $ABC$ (as described in Lemma~\ref{L:firstquadr}) is of \emph{right} type if its shorter sides lie on the coordinate axes. It is determined uniquely up
to isomorphism by its boundary type. In particular, if the types of the shorter sides are $m$ and $j$, then the
hypotenuse is of type $\gcd\{m,j\}$. (See~\cite[\S~4]{MRS11}.) A dyadic triangle $ABC$ is of \emph{hat} type if $C$ lies on the $x$-axis. Note that a hat triangle is the union of two right triangles $ABD$ and $BDC$, where $D$ is the projection of $B$ onto the $x$-axis.

When a vertex $A$ of a dyadic triangle located at the origin is already chosen, the triangle $ABC$ described in Lemma~\ref{L:firstquadr}, may come in one of the three types described in the following proposition.

\begin{proposition}\cite[\S~5,\S~6]{MRS11},\cite[\S~1]{MMR19a}\label{P:classif}
Each dyadic triangle $ABC$ contained in the first quadrant, with $A$ located at the origin,
is isomorphic to a triangle $T_{i,j,m,n}$ of precisely one of the three following types:
\begin{itemize}
\item[(a)] right triangles $T_{0,j,m,0}$ with $j$ and $m$ odd and $j \leq m$;
\item[(b)]  hat triangles $T_{i,j,m,0}$ with $0 < i \leq m/2$, odd $j > 1$, and\\ $\gcd\{i, j\} \neq  j$;
\item[(c)] other triangles, such that neither of $i, n$ is zero, and moreover $j \leq m$, $\gcd\{i, j\} \notin \{i, j, 1\}$ and $\gcd\{m, n\} \notin \{m, n,1\}$.
\end{itemize}
\end{proposition}

For a fixed vertex $A$, the triangles $T_{i,j,m,n}$ described in Proposition~\ref{P:classif} are sometimes called \emph{representative triangles}, of types (a), (b) and (c), respectively.

As we will now restrict our attention to dyadic convex sets, we may sometimes drop the adjective ``dyadic'' when it is clear from the context.

\begin{remark} Proposition~\ref{P:classif} provides a classification, up to $\mathbb D$-module automorphism of the plane $\mbD^2$, of dyadic triangles with one chosen vertex located at the origin.
Let us call such triangles \emph{pointed triangles}. For each dyadic triangle $T$, there are three pointed translated triangles isomorphic to $T$. By Lemma~\ref{L:firstquadr} and Proposition~\ref{P:classif}, each one of these is then  isomorphic to the corresponding representative triangle located as shown in Lemma~\ref{L:firstquadr}. Note, however, that one of the three representative triangles of a triangle $T$ may be of type (c), while another may be of the hat type (b).

To see this, consider a triangle $T_{i,j,m,n}$ of type (c). Recall that the conditions provided in Proposition~\ref{P:classif}(c) guarantee that the points $B$ and $C$ are not axial. This means that there is no $\mathbb D$-module automorphism of the plane $\mbD^2$ which maps any of these points to a point on one of the axes. (See \cite[Lemma~4.11]{MRS11}.) To make the calculations easier, assume additionally that $j > i$ and $m > n$. Now translate the triangle $ABC$ by the vector $(-i,-j)$ to obtain the triangle with vertices $(-i,-j), (0,0), (m-i,n-j)$, then use the symmetry with respect to the axis $Ox$, and finally use the isomorphism of the plane $\mbD^2$ given by the matrix
\[
\begin{bmatrix}
1 & 0 \\
1 & 1
\end{bmatrix}\, .
\]
One obtains the triangle with vertices $A' = (j-i,j), B' = (0,0)$ and $C' = (m-n+j-i,j-n)$. The triangle $A'B'C'$ is contained in the first quadrant, and is not a right triangle. However it may be of type (b) or (c). Indeed, if $B =(i,j) = (12, 15)$ and $C = (m,n) = (15, 12)$, then $A' = (j-i,j) = (3,15)$ and $C' = (m-n+j-i,j-n) = (6,3)$, and both points $A'$ and $C'$ are axial. So for each of these vertices, one can find an affine space automorphism taking it to a point of an axis. For example the automorphism given by the matrix
\[
\begin{bmatrix}
1 & 1 \\
0 & -2
\end{bmatrix}
\]
maps the triangle $A'B'C'$ onto the triangle with vertices $A'' = (3,-27),\\ B'' = (0,0), C'' = (6,0)$. Then the symmetry with respect to the $x$-axis maps the triangle $A''B''C''$ onto the triangle with vertices
$(3,27), (0,0),\\ (6,0)$. The latter triangle is isomorphic to the representative hat triangle $T_{1,9,2,0}$ with vertices $A^o = (1,9), B^o = (0,0), C^o = (2,0)$.

However, if $B = (10,25)$ and $C = (15,9)$, then $A' = (15,25)$ and $C' = (21,16)$, which are not axial. Hence also the (representative) triangle $A'B'C'$ is of type (c).

The triangles obtained are illustrated on Figures~\ref{F:6}, \ref{F:7} and \ref{F:8}.

\begin{figure}[bht]
	\begin{center}
		\begin{picture}(300,200)(0,0)
			
		\put(-10,20){\vector(1,0){150}}
		\put(10,0){\vector(0,1){150}}
		
		\put(90,120){\circle*{5}}
			\put(92,122){$B=(12,15)$}
			
				\put(110,100){\circle*{5}}
				
			\put(116,100){$C=(15,12)$}
		\put(30,18){\line(0,1){4}}
		\put(27,8){$3$}
			\put( 8,40){\line(1,0){4}}
		\put(-2,37){$3$}
			\put(10,20){\circle*{5}}
		\put(0,7){$A$}		
		\put(180,20){\vector(1,0){100}}
\put(200,0){\vector(0,1){150}}

\put(220,120){\circle*{5}}
\put(222,122){$A'=(3,15)$}

\put(240,40){\circle*{5}}

\put(242,42){$C'=(6,3)$}
\put(220,18){\line(0,1){4}}
\put(217,8){$3$}
\put(198,40){\line(1,0){4}}
\put(188,37){$3$}
\put(200,20){\circle*{5}}
\put(187,7){$B'$}
\thicklines

\put(200,20){\line(1,5){20}}
\put(200,20){\line(2,1){40}}
\put(220,120){\line(1,-4){20}}

\put(10,20){\line(4,5){80}}
\put(10,20){\line(5,4){100}}
\put(90,120){\line(1,-1){20}}	
		\end{picture}
	\end{center}
	\caption{}
	\label{F:6}
\end{figure}

\begin{figure}[bht]
	\begin{center}
		\begin{picture}(300,150)(0,0)
		\put(20,0){\vector(0,1){130}}
		\put(0,100){\vector(1,0){100}}
	 	\put(40,98){\line(0,1){4}}
		\put(37,87){$3$}
		\put( 18,120){\line(1,0){4}}
		\put(10,115){$9$}
		\put(20,100){\circle*{5}}
		\put(0,88){$B''$}
		\put(60,100){\circle*{5}}
		\put(70,105){$C''=(6,0)$}
		\put(40,40){\circle*{5}}
		\put(43,24){$A''=(3,-27)$}	
				
		\put(200,20){\vector(1,0){100}}
		\put(220,0){\vector(0,1){130}}
		\put(250,18){\line(0,1){4}}
		\put(247,7){$1$}
		\put( 218, 50){\line(1,0){4}}
		\put(210,45){$3$}
		\put(220,20){\circle*{5}}
		\put(200,8){$B^o$}
		
		\put(280,20){\circle*{5}}
		
		\put(257,8){$C^o=(2,0)$}
		
		\put(250,110){\circle*{5}}
		\put(244,114){$A^o=(1,9)$}		
\thicklines
 		\put(20,100){\line(1,-3){20}}
 \put(60,100){\line(-1,-3){20}}
		\put(220,20){\line(1,3){30}}
		\put(280,20){\line(-1,3){30}}	
		\end{picture}
	\end{center}
	\caption{}
	\label{F:7}
\end{figure}

\begin{figure}[bht]
	\begin{center}
		\begin{picture}(330,200)(0,0)
		
		\put(-10,20){\vector(1,0){170}}
		\put(10,0){\vector(0,1){150}}
		
		\put(70,120){\circle*{5}}
		\put(72,122){$B=(10,25)$}
		
		\put(98,55){\circle*{5}}
		
		\put(103,55){$C=(15,9)$}
		\put(40,18){\line(0,1){4}}
		\put(37,8){$5$}
		\put( 8,40){\line(1,0){4}}
		\put(-2,37){$5$}
		\put(10,20){\circle*{5}}
		\put(-12,7){$A$}		
		\put(170,20){\vector(1,0){170}}
		\put(190,0){\vector(0,1){150}}
		
		\put(250,120){\circle*{5}}
		\put(204,126){$A'=(15,25)$}
		
		\put(283,83){\circle*{5}}
		
		\put(266,63){$C'=(21,16) $}
		
		\put(210,18){\line(0,1){4}}
		\put(207,8){$5$}
		\put(188,40){\line(1,0){4}}
		\put(178,37){$5$}
		\put(190,20){\circle*{5}}
		\put(168,7){$B'$}
		\thicklines
		
		\put(10,20){\line(3,5){60}}
		\put(190,20){\line(3,5){60}}
		\put(190,20){\line(3,2){92}}
		\put(70,120){\line(2,-5){27}}

		\put(250,120){\line(5,-6){30}}
		\put(10,20){\line(5,2){90}}
	
		\end{picture}
	\end{center}
	\caption{}
	\label{F:8}
\end{figure}

\end{remark}

The method used in the proof of Theorem~\ref{P:fingen} shows how to find generators of any representative triangle $T = T_{i,j,m,n}$. One needs generators of the sides and generators of any two-dimensional simplex $\mbS$ contained in $T$. Each side of the triangle $T$ is an interval, and so is generated by $2$ or $3$ elements (two vertices and possibly one additional generator), and the simplex has three generators.  In fact, we do not need all three generators of the simplex. As one of the vertices of $\mbS$, we can take any point $E$ in the interior of $T$ which is a barycentric combination of all three vertices of $T$.
Sometimes, even less number of generators is sufficient.
If $T$ is a right triangle, then the minimal set of generators of its sides provides a minimal set of generators of the triangle. Hence the minimal number of generators of a right triangle lies between $3$ and $6$.
For example,
the right triangle $T=T_{0,3,3,0}$  is minimally generated by the vertices and three additional  points $(0,1),\,(1,0)$ and $(2,1)$ (see Figure ~\ref{T:1}).
	\begin{figure}[bht]
	\begin{center}
		\begin{picture}(140,150)(0,0)

			\put(0,20){\vector(1,0){130}}
			\put(20,0){\vector(0,1){130}}
			\put(50,7){$1$}
			\put(50, 20){\circle*{5}}
			\put(10,45){$1$}
			\put(20, 50){\circle*{5}}
			\put(80, 50){\circle*{5}}
			\put(20, 50){\circle*{5}}
			\put(20,20){\circle*{5}}
			\put(10,10){$0$}
			\put(110,20){\circle*{5}}
			\put(114,25){$ (3,0)$}
			\put(20,110){\circle*{5}}
			\put(25,115){$ (0,3)$}
			\put(20,110){\line(1,-1){90}}
			\thicklines
			\put(20,110){\line(1,-1){90}}
			\put(20,20){\line(1,0){90}}
			\put(20,20){\line(0,1){90}}
			\put(20,50){\line(1,-1){30}}	
		\end{picture}
	\end{center}
	\caption{}
	\label{T:1}
\end{figure}

The hat triangle with vertices $(-1,0), (0,1), (0,3)$ is minimally generated by its vertices and the point $(0,1)$. On the other hand, seven generators are sufficient to generate the hat triangle of the next example.
\begin{example}\label{Ex:hattr}
The hat triangle $T=T_{3,15,6,0}$  is minimally generated by the vertices and four additional  points $(1,5),(5,5),(1,0)$ and $(1,1)$ (see Figure ~\ref{T:2}).
	\begin{figure}[bht]
	\begin{center}
		\begin{picture}(250,150)(0,0)

			\put(0,20){\vector(1,0){220}}
			\put(20,0){\vector(0,1){130}}
			\put(50,7){$1$}
			\put(50, 20){\circle*{5}}
			\put(10,45){$5$}
			\put(18, 50){\line(1,0){4}}
			\put(50, 50){\circle*{5}}
			\put(170, 50){\circle*{5}}
			\put(50, 30){\circle*{5}}
			\put(55,30){$ (1,1)$}
			\put(20,20){\circle*{5}}
			\put(10,10){$0$}
			\put(200,20){\circle*{5}}
			\put(204,25){$ (6,0)$}
			\put(110,110){\circle*{5}}
			\put(115,107){$ (3,15)$}
			\thicklines
			\put(20,20){\line(1,1){90}}
			\put(110,110){\line(1,-1){90}}
			\put(20,20){\line(1,0){180}}
			
			\put(20,20){\line(3,1){30}}
			\put(50,20){\line(0,1){10}}
		\end{picture}
	\end{center}
	\caption{}
	\label{T:2}
\end{figure}

\end{example}
	
For polygons, one may use a similar method as in the case of general dyadic triangles.
A polygon $P$ is generated by the generators of its sides and two generators of a (right) simplex $\mbS$ chosen similarly as in the case of dyadic triangles.

\begin{example}\label{Ex:pol}
Consider the polygon $P$ with the vertices of the triangle of Example~\ref{Ex:hattr} and one additional vertex $(3,-1)$. Then the points $(0,0), (1,0), (6,0)$ and $(3,-1)$ generate the triangle with vertices $(0,0), (6,0), (3,-1)$ and the polygon $P$ is minimaly generated by the generators of the hut triangle and the vertex $(3,-1)$.
\end{example}

\subsection{Dyadic semitriangles}

Next we consider semitriangles, i.e. se\-mi-polytopes closing to triangles. Without lose of generality, we assume that a semitriangle $S$ is a \emph{representative semitriangle}, with  a representative triangle $T = T_{i,j,m,n}$ as its closure. The one-dimensional walls of $S$ will be called its \emph{sides}. A semitriangle $S$ which is not a polytope has at least one side which is not an interval. Each side of $S$ is isomorphic to an interval; hence it is generated by $2$ or $3$ elements. (Recall that a minimal set of generators of $S$ must contain the minimal set of generators of the sides.) To find the remaining generators of $S$, one may use the method described in the proof of Proposition~\ref{P:spolfing}. Since anchors belong to the sides of $S$, they are generated by the generators of the sides. So we only need to add the remaining generators of the inner polytope $P$ of $S$. Since $P$ is a dyadic triangle, the minimal number of its generators may be found using the method described in section~\ref{S:dyadtriangles}.
As we have seen in Example~\ref{Ex:notdpol}, it is sometimes possible to find a smaller set of generators of a semitriangle.
We will show that, as in the case of triangles,  each semitriangle is generated by the generators of its maximal walls and generators of a two-dimensional simplex contained in it.

For any representative semitriangle $S$ in the plane $\mbD^2$, and any point $E$ in the interior of $S$ which is a barycentric combination of the three vertices of $S$, let $\mbS$ be a right two-dimensional simplex contained in $S$ with shorter sides parallel to the axes intersecting at $E$. The following lemma is a version of Lemma~\ref{L:lines}, formulated for semitriangles instead of polytopes.

\begin{lemma}\label{L:linestr}
Let $S$ be a (representative) semitriangle.
Then for any interior point $A \in S$
there exist a point $B$ on a side of $S$, a right (geometric) simplex $\mbS$ as described above, and a point $C$ in $\rm{int}(\mbS)$, such that $A$ belongs to the intersection $l(B,C) \cap S$ of $S$ and the (real) line $l(B,C)$ through $B$ and $C$.
\end{lemma}
\begin{proof}
First note that each side of $S$ is contained in a real line going through two (dyadic) vertices of $S$. By the Density Lemma~\ref{L:points}, for any two points of this line there is a dyadic point located between them. We will show that such a point may be located on the appropriate side of $S$. Without loss of generality assume that a side $W$ of $S$ is contained in the (real) line $l$ given by an equation $y = ax$ for some dyadic number $a$, and going through the vertices of $W$. Note that there is a positive odd number $m$ such that $l \cap S = \{(d,ad) \mid d \in m\mbD\}$. Now if $R_1 = (r,ar)$ and $R_2 = (s,as)$  are any two distinct points of $l$ with $r < s$, then there are natural numbers $k$ and $n$ such that $r < km/2^n < s$. So if the points $R_1$ and $R_2$ lie between the vertices of $S$, then the point $B = (km/2^n, akm/2^n)$ belongs to the side of $S$. Let $\mbS$ be a right two-dimensional simplex contained in $S$ with shorter sides parallel to the axes, as described above, and such that the real line $l(A,B)$ has a nonempty intersection with $\mr{int}(\mbS_r)$. By the Density Lemma, there is a point $C$ on the line $l(A,B)$ located between any two (real) points of $l(A,B) \cap \mr{int}(\mbS_r)$.
\end{proof}

\begin{corollary}
Each semitriangle is generated by the generators of the sides and two vertices of a two-dimensional right simplex
contained in the  semitriangle.
\end{corollary}

Note that if $S$ is a (representative) semitriangle closing to a right triangle $ABC$ with the interval $[A,C]$ as the side of $S$, then $S$ is generated by the generators of the sides and only one additional generator of a right simplex $\mbS$ with one of its shorter sides on the $x$-axis. We have a similar situation if $S$ is a semitriangle closing to a hat triangle.

\begin{example}
Consider the semitriangle $S$ closing the hat triangle of Example~\ref{Ex:hattr} consisting of the interior of the hat triangle, the whole interval joining $(0,0)$ and $(6,0)$, the subgroupoid generated by $(0,0)$ and $(3,15)$, and the subgroupoid generated by $(6,0)$ and $(3,15)$. Then $S$ is generated by the three vertices and the points $(1,0)$ and $(1,1)$.
\end{example}

\subsection*{Acknowledgement}

The authors are grateful to anonymous referees for helpful comments on an earlier version of this paper.

\end{document}